\documentclass[11pt,reqno]{amsart}
\usepackage{amssymb}
\usepackage{amsmath}
\usepackage{amsfonts}

\newtheorem{theorem}{Theorem}
\theoremstyle{plain}

\newtheorem{corollary}{Corollary}

\newtheorem{definition}{Definition}

\newtheorem{lemma}{Lemma}

\newtheorem{proposition}{Proposition}

\numberwithin{equation}{section}
\numberwithin{equation}{section}
\input{tcilatex}

\begin{document}
\title[Beahviors of the energy]{Energy decay rates for solutions of the wave
equation with linear damping in exterior domain}
\author{M. Daoulatli}
\address{Department of Mathematics, FSB, University of Carthage \& LAMSIN,
ENIT, University of Tunis Elmanar}
\email[M. Daoulatli]{moez.daoulatli@infcom.rnu.tn}
\date{\today }
\thanks{This work is supported by the Tunisian Ministry for Scientific
Research and Technology within the LAB-STI 02 program.}

\begin{abstract}
In this paper we study the behavior of the energy and the $L^{2}$ norm of
solutions of the wave equation with localized linear damping in exterior
domain. Let $u$ be a solution of the wave system with initial data $\left(
u_{0},u_{1}\right) $. We assume that the damper is positive at infinity then
under the Geometric Control Condition of Bardos et al \cite{blr} (1992), we
prove that:

\begin{enumerate}
\item The total energy $E_{u}\left( t\right) \leq C_{0}\left( 1+t\right)
^{-1}I_{0}$ and $\left\Vert u\left( t\right) \right\Vert _{L^{2}}^{2}\leq
C_{0}I_{0}$ if $\left( u_{0},u_{1}\right) $ belong to $H_{0}^{1}\left(
\Omega \right) \times L^{2}\left( \Omega \right) ,$ where 
\begin{equation*}
I_{0}=\left\Vert u_{0}\right\Vert _{H^{1}}^{2}+\left\Vert u_{1}\right\Vert
_{L^{2}}^{2}.
\end{equation*}

\item The total energy $E_{u}\left( t\right) \leq C_{2}\left( 1+t\right)
^{-2}I_{1}$ and $\left\Vert u\left( t\right) \right\Vert _{L^{2}}^{2}\leq
C_{2}\left( 1+t\right) ^{-1}I_{1},$ if the initial data $\left(
u_{0},u_{1}\right) $ belong to $H_{0}^{1}\left( \Omega \right) \times
L^{2}\left( \Omega \right) $ \ and verifies $\left\Vert d\left( \cdot
\right) \left( u_{1}+au_{0}\right) \right\Vert _{L^{2}}<+\infty ,$ where 
\begin{equation*}
I_{1}=\left\Vert u_{0}\right\Vert _{H^{1}}^{2}+\left\Vert u_{1}\right\Vert
_{L^{2}}^{2}+\left\Vert d\left( \cdot \right) \left( u_{1}+au_{0}\right)
\right\Vert _{L^{2}}^{2}.
\end{equation*}%
$.$
\end{enumerate}
\end{abstract}

\maketitle

\section{Introduction and Statement of the result}

Let $O$ be a compact domain of $%
\mathbb{R}
^{d}$ $\left( d\geq 2\right) $ with $C^{\infty }$ boundary $\Gamma =\partial
\Omega $ and $\Omega =\mathbb{R}^{d}\backslash O$. Consider the following
wave equation with localized linear damping 
\begin{equation}
\left\{ 
\begin{array}{lc}
\partial _{t}^{2}u-\Delta u+a\left( x\right) \partial _{t}u=0 & \text{in }%
\mathbb{R}_{+}\times \Omega , \\ 
u=0 & \text{on }\mathbb{R}_{+}\times \Gamma , \\ 
u\left( 0,x\right) =u_{0}\quad \text{ and }\quad \partial _{t}u\left(
0,x\right) =u_{1}. & 
\end{array}%
\right.  \label{system}
\end{equation}%
Here $\Delta $ denotes the Laplace operator in the space variables. $a\left(
x\right) $ is a nonnegative function in $L^{\infty }\left( \Omega \right) $.

Let 
\begin{equation*}
A=\left( 
\begin{array}{cc}
0 & I \\ 
\Delta & -a%
\end{array}%
\right) ,
\end{equation*}%
and $H=H_{D}\left( \Omega \right) \times L^{2}\left( \Omega \right) $, the
completion of $(C_{0}^{\infty }\left( \Omega \right) )^{2}$\ with respect to
the norme%
\begin{equation*}
\left\Vert (\varphi _{0},\varphi _{1})\right\Vert _{H}^{2}=\frac{1}{2}%
\int_{\Omega }\left\vert \nabla \varphi _{0}\right\vert ^{2}+\left\vert
\varphi _{1}\right\vert ^{2}dx,
\end{equation*}%
then the domain of $A$%
\begin{equation*}
D\left( A\right) =\left\{ \left( u_{0},u_{1}\right) \in H,A\left( 
\begin{array}{c}
u_{0} \\ 
u_{1}%
\end{array}%
\right) \in H\right\} .
\end{equation*}%
Let $n\in 
\mathbb{N}
$ and $\left( u_{0},u_{1}\right) \in D\left( A^{n}\right) $. Linear
semigroup theory applied to (\ref{system}), provides existence of a unique
solution $u$ in the class%
\begin{equation*}
\left( u,\partial _{t}u\right) \in C^{k}\left( 
\mathbb{R}
_{+},D\left( A^{n-k}\right) \right) ,\text{ with }k\leq n.
\end{equation*}

Moreover, if $\left( u_{0},u_{1}\right) $ is in $H_{0}^{1}\left( \Omega
\right) \times L^{2}\left( \Omega \right) $, then the system (\ref{system}),
admits a unique solution $u$ in the class%
\begin{equation*}
u\in C^{0}\left( 
\mathbb{R}
_{+},H_{0}^{1}\left( \Omega \right) \right) \cap C^{1}\left( 
\mathbb{R}
_{+},L^{2}\left( \Omega \right) \right) .
\end{equation*}%
With $\left( \text{\ref{sys:nonlinear}}\right) $ we associate the energy
functional given by%
\begin{equation*}
E_{u}\left( t\right) =\frac{1}{2}\int_{\Omega }\left( \left\vert \nabla
u\left( t,x\right) \right\vert ^{2}+\left\vert \partial _{t}u\left(
t,x\right) \right\vert ^{2}\right) dx.
\end{equation*}%
The energy functional satisfies the following identity%
\begin{equation}
E_{u}\left( T\right) +\int_{0}^{T}\int_{\Omega }a\left( x\right) \left\vert
\partial _{t}u\right\vert ^{2}dxdt=E_{u}\left( 0\right) ,
\label{energy inequality}
\end{equation}%
for every $T\geq 0$.

Zuazua \cite{zuazua}, Nakao \cite{nakao klein}, Dehman et al \cite{dlz} and
Aloui et al \cite{aloui-b-n}\ have considered the problem for the
Klein-Gordon type wave equations with localized dissipations. For the
Klein-Gordon equations the energy functional it self contains the $L^{2}$
norm and boundedness of $L^{2}$ norm of solution is trivial. Thus under a
geometric condition we can show that the energy decays exponentially while
for the system $\left( \ref{system}\right) $ the energy decay rate is weaker
and more delicate.

In the case when $a\left( x\right) \geq \epsilon _{0}>0$ in all of $\Omega $
we know that 
\begin{equation}
E_{u}\left( t\right) \leq C_{0}\left( 1+t\right) ^{-1}I_{0}\text{ and }%
\left\Vert u\left( t\right) \right\Vert _{L^{2}}^{2}\leq C_{0}I_{0},\text{
for all }t\geq 0,  \label{energy rate 1}
\end{equation}%
for weak solution $u$ to the system $\left( \ref{system}\right) $ with
initial data in $H_{0}^{1}\left( \Omega \right) \times L^{2}\left( \Omega
\right) $.

Nakao in \cite{nakao} obtained the same estimates in $\left( \ref{energy
rate 1}\right) $ for a damper $a$ which is positive near some part of the
boundary (Lions's condition) and near infinity.

On the other hand, Dan-Shibata \cite{dan-shi} studied the local energy decay
estimates for the compactly supported weak solutions of $\left( \ref{system}%
\right) $ with $a\left( x\right) =1$%
\begin{equation}
\int_{\Omega \cap B_{R}}\left( \left\vert \nabla u\left( t,x\right)
\right\vert ^{2}+\left\vert \partial _{t}u\left( t,x\right) \right\vert
^{2}\right) dx\leq C\left( 1+t\right) ^{-d},  \label{local energy estimate}
\end{equation}%
where $B_{R}=\left\{ x\in 
\mathbb{R}
^{d},\left\vert x\right\vert <R\right\} $.

Furthermore Ikehata and Matsuyama in \cite{ike-matsu} obtained a more
precise decay estimate for the energy of solutions of the problem $\left( %
\ref{system}\right) $ with $a\left( x\right) =1$ and for weighted initial
data%
\begin{equation}
E_{u}\left( t\right) \leq C_{2}\left( 1+t\right) ^{-2}I_{1}\text{ and }%
\left\Vert u\left( t\right) \right\Vert _{L^{2}}^{2}\leq C_{2}\left(
1+t\right) ^{-1}I_{1}\text{ for all }t\geq 0.  \label{energy rate 2}
\end{equation}%
Especially this estimate seems sharp for $d=2$ as compared with that of \cite%
{dan-shi}.

Ikehata in \cite{ikehata} derived a fast decay rate like $\left( \ref{energy
rate 2}\right) $ for solutions of the system $\left( \ref{system}\right) $
with weighted initial data and assuming that $a\left( x\right) \geq \epsilon
_{0}>0$ at infinity and $O=%
\mathbb{R}
^{d}\backslash \Omega $ is star shaped with respect to the origin.

For another type of total energy decay property we refer the reader to \cite%
{ikehata 1,kawa,Racke,nakao-ba} and reference therein.

Before introducing our results we shall state several assumptions:

\begin{description}
\item[Hyp A] There exists $L>0$ such that 
\begin{equation*}
a\left( x\right) >\epsilon _{0}>0\text{ for }\left\vert x\right\vert \geq L.
\end{equation*}
\end{description}

\begin{definition}
$\left( \omega ,T\right) $ geometrically controls $\Omega $, i.e. every
generalized geodesic travelling with speed $1$ and issued at $t=0$, enters
the set $\omega $ in a time $t<T$.
\end{definition}

This condition is called Geometric Control Condition (see e.g.\cite{blr} ).
We shall relate the open subset $\omega $ with the damper $a$ by%
\begin{equation*}
\omega =\left\{ x\in \Omega ;a\left( x\right) >\epsilon _{0}>0\right\} .
\end{equation*}

We note that according to \cite{blr} and \cite{bg}\ the Geometric Control
Condition of Bardos et al is a necessary and sufficient condition for the
stabilization of the wave equation in bounded domain.

The goal of this paper is to prove that under the geometric control
condition of Bardos et al \cite{blr} and for a damper $a$ positive near
infinity, the estimates in $\left( \ref{energy rate 1}\right) $ hold for all
solutions of the system $\left( \ref{system}\right) $ with initial data in $%
H_{0}^{1}\left( \Omega \right) \times L^{2}\left( \Omega \right) $ and to
show that the estimates in $\left( \ref{energy rate 2}\right) $ hold for all
solutions of the system $\left( \ref{system}\right) $ with weighted initial
data. Moreover we show that for every $p\in 
\mathbb{N}
^{\ast }$ there exists a initial data in $H_{0}^{1}\left( \Omega \right)
\times L^{2}\left( \Omega \right) $ such that the solution $v$ of $\left( %
\ref{system}\right) $ verifies%
\begin{equation*}
E_{v}\left( t\right) \leq C\left( 1+t\right) ^{-p}\text{ and }\left\Vert
v\left( t\right) \right\Vert _{L^{2}}^{2}\leq C\left( 1+t\right) ^{-p+1},%
\text{ for all }t\geq 0
\end{equation*}%
and for some $C>0$ depending on the initial data.

\begin{theorem}
We assume that Hyp A holds and $(\omega $,$T)$ geometrically controls $%
\Omega $. Then there exists $C_{0}>0$ such that the following estimates%
\begin{equation*}
E_{u}\left( t\right) \leq C_{0}\left( 1+t\right) ^{-1}I_{0}\text{ and }%
\left\Vert u\left( t\right) \right\Vert _{L^{2}}^{2}\leq C_{0}I_{0},\text{
for all }t\geq 0,
\end{equation*}%
hold for every solution $u$ of $\left( \ref{system}\right) $ with initial
data $\left( u_{0},u_{1}\right) $ in $H_{0}^{1}\left( \Omega \right) \times
L^{2}\left( \Omega \right) $, where%
\begin{equation*}
I_{0}=\left\Vert u_{0}\right\Vert _{H^{1}}^{2}+\left\Vert u_{1}\right\Vert
_{L^{2}}^{2}.
\end{equation*}
\end{theorem}

As a corollary of theorem 1 we have:

\begin{proposition}
Let $n\in 
\mathbb{N}
^{\ast }$. We assume that Hyp A holds and $(\omega $,$T)$ geometrically
controls $\Omega $. Let $\left( u_{0},u_{1}\right) $ in $D\left(
A^{n}\right) $, such that $u_{0}\in L^{2}\left( \Omega \right) $. Then the
solution $u$ of $\left( \ref{system}\right) $ satisfies 
\begin{equation*}
E_{\partial _{t}^{n}u}\left( t\right) \leq C_{n}\left( 1+t\right)
^{-n-1}I_{0,n}\text{ for all }t\geq 0,
\end{equation*}%
\begin{equation*}
\left\Vert \partial _{t}^{n}u\left( t\right) \right\Vert _{L^{2}}^{2}\leq
C_{n-1}\left( 1+t\right) ^{-n}I_{0,n-1}\text{ for all }t\geq 0,\text{ }
\end{equation*}
and%
\begin{equation*}
\left\Vert \Delta \partial _{t}^{n-1}u\left( t\right) \right\Vert
_{L^{2}}^{2}\leq C_{n}\left( 1+t\right) ^{-n}I_{0,n}\text{ for all }t\geq 0.%
\text{ }
\end{equation*}%
where $C_{p}$ is a positive constant independent of the initial data and%
\begin{equation*}
I_{1,p}=\sum_{i=0}^{p}\left\Vert A^{i}\left( u_{0},u_{1}\right) \right\Vert
_{H}^{2}+\left\Vert u_{0}\right\Vert _{L^{2}}^{2},\text{ for }p\in 
\mathbb{N}
.
\end{equation*}
\end{proposition}

In the sequel, we use 
\begin{equation*}
d\left( x\right) =\left\{ 
\begin{array}{lc}
\left\vert x\right\vert & d\geq 3, \\ 
\left\vert x\right\vert \ln \left( B\left\vert x\right\vert \right) & d=2,%
\end{array}%
\right.
\end{equation*}%
with $B\underset{x\in \Omega }{\inf }\left\vert x\right\vert \geq 2$.

\begin{theorem}
We assume that Hyp A holds and $(\omega $,$T)$ geometrically controls $%
\Omega $. Then there exists $C_{2}>0$ such that the following estimates%
\begin{equation*}
E_{u}\left( t\right) \leq C_{2}\left( 1+t\right) ^{-2}I_{1}\text{ and }%
\left\Vert u\left( t\right) \right\Vert _{L^{2}}^{2}\leq C_{2}\left(
1+t\right) ^{-1}I_{1}\text{ for all }t\geq 0,
\end{equation*}%
hold for every solution $u$ of $\left( \ref{system}\right) $ with initial
data $\left( u_{0},u_{1}\right) $ in $H_{0}^{1}\left( \Omega \right) \times
L^{2}\left( \Omega \right) $ which satisfies%
\begin{equation*}
\left\Vert d\left( \cdot \right) \left( u_{1}+au_{0}\right) \right\Vert
_{L^{2}}<+\infty ,
\end{equation*}%
where%
\begin{equation*}
I_{1}=\left\Vert u_{0}\right\Vert _{H^{1}}^{2}+\left\Vert u_{1}\right\Vert
_{L^{2}}^{2}+\left\Vert d\left( \cdot \right) \left( u_{1}+au_{0}\right)
\right\Vert _{L^{2}}^{2}.
\end{equation*}
\end{theorem}

As a corollary we have:

\begin{proposition}
Let $n\in 
\mathbb{N}
^{\ast }$. We assume that Hyp A holds and $(\omega $,$T)$ geometrically
controls $\Omega $. Let $\left( u_{0},u_{1}\right) $ in $D\left(
A^{n}\right) $, such that $u_{0}\in L^{2}\left( \Omega \right) $ and 
\begin{equation*}
\left\Vert d\left( \cdot \right) \left( u_{1}+au_{0}\right) \right\Vert
_{L^{2}}<+\infty 
\end{equation*}%
Then the solution $u$ of $\left( \ref{system}\right) $ satisfies 
\begin{equation*}
E_{\partial _{t}^{n}u}\left( t\right) \leq C_{n}\left( 1+t\right)
^{-n-2}I_{1,n},\text{ for all }t\geq 0,
\end{equation*}%
\begin{equation*}
\left\Vert \partial _{t}^{n}u\left( t\right) \right\Vert _{L^{2}}^{2}\leq
C_{n-1}\left( 1+t\right) ^{-n-1}I_{1,n-1},\text{ for all }t\geq 0,\text{ }
\end{equation*}%
and 
\begin{equation*}
\left\Vert \Delta \partial _{t}^{n-1}u\left( t\right) \right\Vert
_{L^{2}}^{2}\leq C_{n+1}\left( 1+t\right) ^{-n-1}I_{1,n},\text{ for all }%
t\geq 0.
\end{equation*}%
where $C_{p}$ is a positive constant independent of the initial data and%
\begin{equation*}
I_{1,p}=\sum_{i=0}^{p}\left\Vert A^{i}\left( u_{0},u_{1}\right) \right\Vert
_{H}^{2}+\left\Vert u_{0}\right\Vert _{L^{2}}^{2}+\left\Vert d\left( \cdot
\right) \left( u_{1}+au_{0}\right) \right\Vert _{L^{2}}^{2},\text{ for }p\in 
\mathbb{N}
.
\end{equation*}
\end{proposition}

\section{Proof of Theorem 1}

In order to prove theorem 1 we need some preliminary results.

\begin{proposition}
We assume that Hyp A holds and $(\omega $,$T)$ geometrically controls $%
\Omega $. Let $M=\Omega \cap B_{2L}$. Setting 
\begin{equation*}
\omega _{1}=\omega \cap B_{2L}=\left\{ x\in \Omega \cap B_{2L};a\left(
x\right) >\epsilon _{0}>0\right\} .
\end{equation*}%
Then $\left( \omega _{1},T\right) $ geometrically controls $M$. Therefore
there exists a positive constant $C_{T}^{1}$, such that the following
estimate%
\begin{equation}
E_{w}\left( t\right) \leq C_{T}^{1}\left( \int_{t}^{t+T}\int_{M}a\left\vert
\partial _{t}w\right\vert ^{2}+\left\vert f\left( s,x\right) \right\vert
^{2}dxds\right) ,  \label{observability bounded}
\end{equation}%
holds for every $t\geq 0$, for every solution $w$ of 
\begin{equation}
\left\{ 
\begin{array}{ll}
\partial _{t}^{2}w-\Delta w+a\left( x\right) \partial _{t}w=f\left(
t,x\right) & 
\mathbb{R}
_{+}\times M, \\ 
w=0 & 
\mathbb{R}
_{+}\times \partial M, \\ 
\left( w\left( 0\right) ,\partial _{t}w\left( 0\right) \right) =\left(
w_{0},w_{1}\right) , & 
\end{array}%
\right.  \label{sys:nonlinear}
\end{equation}%
with initial data in the energy space $H_{0}^{1}\left( M\right) \times
L^{2}\left( M\right) $, and for every $f$ in $L_{loc}^{2}\left( 
\mathbb{R}
_{+},L^{2}\left( M\right) \right) $.
\end{proposition}

\begin{proof}
We remind that 
\begin{equation*}
a\left( x\right) >\epsilon _{0}>0\text{ for }\left\vert x\right\vert \geq L.
\end{equation*}%
Let 
\begin{equation*}
\omega _{1}=\left\{ x\in \Omega \cap B_{2L};a\left( x\right) >\epsilon
_{0}>0\right\} .
\end{equation*}%
To prove that $\left( \omega _{1},T\right) $ geometrically controls $M$, we
have only to show that every geodesic starting from $B_{L}$ enters the set $%
\omega _{1}$ in a time $t<T$. Let $\gamma $ a geodesic starting from $B_{L}$%
, then we have the following two cases

\begin{itemize}
\item $\gamma $ stays in $B_{L}$ for every $t\in \left[ 0,T\right[ $. Since $%
\left( \omega =\left\{ x\in \Omega ;a\left( x\right) >\epsilon
_{0}>0\right\} ,T\right) $ geometrically controls $\Omega $, we infer that $%
\gamma $ enters the set $\left\{ x\in \Omega \cap B_{L};a\left( x\right)
>\epsilon _{0}>0\right\} \subset $ $\omega _{1}$.

\item $\gamma $ leaves the ball $B_{L}$ for some $t\in \left[ 0,T\right[ $.
Therefore $\gamma $ enters the set $\left\{ x\in \Omega \cap
B_{2L};\left\vert x\right\vert \geq L\right\} $, since 
\begin{equation*}
a\left( x\right) >\epsilon _{0}>0\text{ for }\left\vert x\right\vert \geq L
\end{equation*}%
we deduce that $\gamma $ enters the set $\omega _{1}$.
\end{itemize}

Therefore using \cite[proposition 3]{daou}, we obtain $\left( \ref%
{observability bounded}\right) $.
\end{proof}

\begin{proposition}
\label{proposition onservability global}We assume that Hyp A holds and $%
(\omega $,$T)$ geometrically controls $\Omega $. Let $\delta >0$ and $\chi
\in C_{0}^{\infty }\left( 
\mathbb{R}
^{d}\right) $. There exists $C_{T,\delta ,\chi }>0$, such that the following
inequality%
\begin{equation}
\int_{t}^{t+T}\int_{\Omega }\chi ^{2}\left( x\right) \left( \left\vert
\nabla u\right\vert ^{2}+\left\vert \partial _{t}u\right\vert ^{2}\right)
dxds\leq C_{T,\delta ,\chi }\left( \int_{t}^{t+T}\int_{\Omega }a\left(
x\right) \left\vert \partial _{t}u\right\vert ^{2}dxds\right) +\delta
E_{u}\left( t\right) ,  \label{observability gradient1}
\end{equation}%
holds for every $t\geq 0$ and for all $u$ solution of $\left( \ref{system}%
\right) $ with initial data $\left( u_{0},u_{1}\right) $ in $H$.
\end{proposition}

\begin{proof}
To prove this result we argue by contradiction: If $\left( \ref%
{observability gradient1}\right) $ was false, there would exist a sequence
of numbers $\left( t_{n}\right) $ and a sequence of solutions $\left(
u_{n}\right) $ such that 
\begin{equation}
\int_{t_{n}}^{t_{n}+T}\int_{\Omega }\chi ^{2}\left( x\right) \left(
\left\vert \nabla u_{n}\right\vert ^{2}+\left\vert \partial
_{t}u_{n}\right\vert ^{2}\right) dxds\geq n\left(
\int_{t_{n}}^{t_{n}+T}\int_{\Omega }a\left( x\right) \left\vert \partial
_{t}u_{n}\right\vert ^{2}dxdt\right) +\delta E_{u_{n}}\left( t_{n}\right) .
\label{proof proposition 4 contradiction estimate}
\end{equation}%
Setting%
\begin{equation*}
\lambda _{n}^{2}=\int_{t_{n}}^{t_{n}+T}\int_{\Omega }\chi ^{2}\left(
x\right) \left( \left\vert \nabla u_{n}\right\vert ^{2}+\left\vert \partial
_{t}u_{n}\right\vert ^{2}\right) dxds\text{ and }v_{n}=\frac{u_{n}\left(
t_{n}+\cdot \right) }{\lambda _{n}}.
\end{equation*}%
It is clear that $\left( \ref{proof proposition 4 contradiction estimate}%
\right) ,$ gives%
\begin{equation}
\int_{0}^{T}\int_{\Omega }a\left( x\right) \left\vert \partial
_{t}v_{n}\right\vert ^{2}dxdt\underset{n\rightarrow +\infty }{%
\longrightarrow }0\text{ and }E_{v_{n}}\left( 0\right) \leq \frac{1}{\delta }%
.  \label{null limit consequnce}
\end{equation}%
Let $Z_{n}$ be the solution of the following system%
\begin{equation*}
\left\{ 
\begin{array}{ll}
\partial _{t}^{2}Z_{n}-\Delta Z_{n}=0 & 
\mathbb{R}
_{+}\times \Omega , \\ 
Z_{n}=0 & 
\mathbb{R}
_{+}\times \Gamma , \\ 
\left( Z_{n}\left( 0\right) ,\partial _{t}Z_{n}\left( 0\right) \right) =%
\frac{1}{\lambda _{n}}\left( u_{n}\left( t_{n}\right) ,\partial
_{t}u_{n}\left( t_{n}\right) \right) . & 
\end{array}%
\right. 
\end{equation*}%
The hyperbolic energy inequality gives%
\begin{equation*}
\underset{\left[ 0,T\right] }{\sup }E_{v_{n}-Z_{n}}^{1/2}\left( s\right)
\leq 2\left\Vert a\left( x\right) \partial _{t}v_{n}\right\Vert
_{L^{1}\left( \left[ 0,T\right] ,L^{2}\left( \Omega \right) \right) },
\end{equation*}%
Now using $\left( \ref{null limit consequnce}\right) $, we infer that \ 
\begin{equation}
\underset{\left[ 0,T\right] }{\sup }E_{v_{n}-Z_{n}}\left( s\right) \underset{%
n\rightarrow +\infty }{\longrightarrow }0.  \label{energy difference}
\end{equation}%
On the other hand,%
\begin{equation*}
\frac{1}{2}\int_{0}^{T}\int_{\omega }\left\vert \partial
_{t}Z_{n}\right\vert ^{2}dxdt\leq \underset{\left[ 0,T\right] }{2T\sup }%
E_{v_{n}-Z_{n}}\left( s\right) +\int_{0}^{T}\int_{\omega }\left\vert
\partial _{t}v_{n}\right\vert ^{2}dxdt.
\end{equation*}%
Since $a\left( x\right) \geq \epsilon _{0}>0$ on $\omega $, from $\left( \ref%
{null limit consequnce}\right) $, we deduce that%
\begin{equation*}
\int_{0}^{T}\int_{\omega }\left\vert \partial _{t}v_{n}\right\vert ^{2}dxdt%
\underset{n\rightarrow +\infty }{\longrightarrow }0.
\end{equation*}%
\ $\left( \ref{energy difference}\right) $ combined with the result above,
gives%
\begin{equation}
\int_{0}^{T}\int_{\omega }\left\vert \partial _{t}Z_{n}\right\vert ^{2}dxdt%
\underset{n\rightarrow +\infty }{\longrightarrow }0.
\label{null limit consequnceZ}
\end{equation}%
It is clear that 
\begin{equation*}
\underset{\left[ 0,T\right] }{\sup }E_{Z_{n}}\left( t\right)
=E_{Z_{n}}\left( 0\right) =E_{v_{n}}\left( 0\right) \leq \frac{1}{\delta }.
\end{equation*}%
Therefore, along a subsequence,$\left( Z_{n}\right) $ is convergent to a
function 
\begin{equation*}
L^{2}\in C\left( \left[ 0,T\right] ;H_{D}\left( \Omega \right) \right) \text{
and }\partial _{t}Z\in L^{2}\left( \left[ 0,T\right] ;L^{2}\left( \Omega
\right) \right) ,
\end{equation*}%
with respect to the weak topology. Since $Z$ satisfies%
\begin{equation*}
\left\{ 
\begin{array}{ll}
\partial _{t}^{2}Z-\Delta Z=0 & \text{in }\left[ 0,T\right] \times \Omega ,
\\ 
Z=0 & \text{on }\left[ 0,T\right] \times \Gamma , \\ 
\left( Z_{0},Z_{1}\right) \in H_{D}\left( \Omega \right) \times L^{2}\left(
\Omega \right)  &  \\ 
\partial _{t}Z\left( t,x\right) =0 & \text{on }\left[ 0,T\right] \times
\omega .%
\end{array}%
\right. 
\end{equation*}%
Therefore we have 
\begin{equation*}
Z\in C\left( \left[ 0,T\right] ;H_{D}\left( \Omega \right) \right) \text{
and }\partial _{t}Z\in C\left( \left[ 0,T\right] ;L^{2}\left( \Omega \right)
\right) ,
\end{equation*}%
We remind that $\left\{ x\in 
\mathbb{R}
^{d},\text{ }\left\vert x\right\vert \geq L\right\} \subset \omega $. By a
classical result of unique continuation, we obtain that $\partial
_{t}Z\equiv 0$ on $\left[ 0,T\right] \times \Omega $. this mean that $%
Z\left( t,x\right) =Z\left( x\right) $ is independent of $t$. Therefore, we
have 
\begin{equation*}
\begin{array}{l}
\Delta Z=0\text{ and }Z\in H_{D}\left( \Omega \right) ,%
\end{array}%
\end{equation*}%
we conclude from this that $Z\equiv 0$ on $\left[ 0,T\right] \times \Omega $
(cf. \cite[theorem 2.2 p 145]{lax-phil}).

The sequence $\left( Z_{n}\right) $ is bounded in $H_{loc}^{1}\left( \left(
0,T\right) \times \Omega \right) $, so eventually after extracting a
subsequence we can associate to the sequence $(Z_{n})$ a microlocal defect
measure $\mu $. This measure satisfies these properties: The support of $\mu 
$ is contained in the characteristic set of the wave operator and it
propagates along the geodesics of $\Omega $.

Now using $\left( \ref{null limit consequnceZ}\right) $, we deduce that 
\begin{equation*}
\mu =0,\text{ on }\left( 0,T\right) \times \omega .
\end{equation*}%
Since $(\omega $,$T)$ geometrically controls $\Omega $ and $\mu $ propagates
along geodesic flow, therefore we obtain 
\begin{equation*}
\mu =0,\text{ on }\left( 0,T\right) \times \Omega .
\end{equation*}%
Let $R>0$ such that the support of $\chi $ is contained in $B_{R}$ and $%
\varphi \in C_{0}^{\infty }\left( 0,T\right) $ such that%
\begin{equation*}
\int_{0}^{T}\varphi \left( s\right) ds\geq \epsilon >0.
\end{equation*}%
Using the finite speed propagation property, we obtain%
\begin{equation*}
\int_{0}^{T}\varphi \left( s\right) \int_{\Omega \cap B_{R+T}}\left\vert
\nabla Z_{n}\left( s\right) \right\vert ^{2}+\left\vert \partial
_{t}Z_{n}\left( s\right) \right\vert ^{2}dxds\geq \epsilon \int_{\Omega \cap
B_{R}}\left\vert \nabla Z_{n}\left( t\right) \right\vert ^{2}+\left\vert
\partial _{t}Z_{n}\left( t\right) \right\vert ^{2}dx,
\end{equation*}%
for all $t\in \left[ 0,T\right] .$ Now passing to the limit and using the
fact that%
\begin{equation*}
\mu =0,\text{ on }\left( 0,T\right) \times \Omega ,
\end{equation*}%
we deduce that%
\begin{equation*}
\int_{\Omega \cap B_{R}}\left\vert \nabla Z_{n}\left( t\right) \right\vert
^{2}+\left\vert \partial _{t}Z_{n}\left( t\right) \right\vert ^{2}dx\underset%
{n\rightarrow +\infty }{\longrightarrow }0,\text{ for all }t\in \left[ 0,T%
\right] .
\end{equation*}%
Using the result above and $\left( \ref{energy difference}\right) $\ we
deduce that%
\begin{equation*}
\int_{\Omega \cap B_{R}}\left\vert \nabla v_{n}\left( t\right) \right\vert
^{2}+\left\vert \partial _{t}v_{n}\left( t\right) \right\vert ^{2}dx\underset%
{n\rightarrow +\infty }{\longrightarrow }0,\text{ for all }t\in \left[ 0,T%
\right] .
\end{equation*}%
So we conclude that%
\begin{equation*}
\int_{\Omega }\chi ^{2}\left( x\right) \left( \left\vert \nabla v_{n}\left(
t\right) \right\vert ^{2}+\left\vert \partial _{t}v_{n}\left( t\right)
\right\vert ^{2}\right) dx\underset{n\rightarrow +\infty }{\longrightarrow }%
0,\text{ for all }t\in \left[ 0,T\right] .
\end{equation*}%
The fact that the energy of $v_{n}$ is decreasing, gives%
\begin{equation*}
\int_{\Omega }\chi ^{2}\left( x\right) \left( \left\vert \nabla v_{n}\left(
t\right) \right\vert ^{2}+\left\vert \partial _{t}v_{n}\left( t\right)
\right\vert ^{2}\right) dx\leq C,\text{ for all }t\in \left[ 0,T\right] .
\end{equation*}%
By the dominated convergence theorem we infer that%
\begin{equation*}
1=\int_{0}^{T}\int_{\Omega }\chi ^{2}\left( x\right) \left( \left\vert
\nabla v_{n}\left( t\right) \right\vert ^{2}+\left\vert \partial
_{t}v_{n}\left( t\right) \right\vert ^{2}\right) dxdt\underset{n\rightarrow
+\infty }{\longrightarrow }0.
\end{equation*}
\end{proof}

As a corollary we have,

\begin{corollary}
We assume that Hyp A holds and $(\omega $,$T)$ geometrically controls $%
\Omega $. Let $\delta $,$R>0$. There exists $C_{T,\delta ,R}>0$, such that
the following inequality%
\begin{equation}
\int_{t}^{t+T}\int_{\Omega \cap B_{R}}\left( \left\vert \nabla u\right\vert
^{2}+\left\vert \partial _{t}u\right\vert ^{2}\right) dxds\leq C_{T,\delta
,R}\left( \int_{t}^{t+T}\int_{\Omega }a\left( x\right) \left\vert \partial
_{t}u\right\vert ^{2}dxds\right) +\delta E_{u}\left( t\right) ,
\label{observability gradient}
\end{equation}%
holds for every $t\geq 0$ and for all $u$ solution of $\left( \ref{system}%
\right) $ with initial data $\left( u_{0},u_{1}\right) $ in $H$.
\end{corollary}

In order to prove theorem 1 we need the following result

\begin{lemma}
\label{lemma xt}Let $\psi \in C_{0}^{\infty }\left( 
\mathbb{R}
^{d}\right) $ such that $0\leq \psi \leq 1$ and%
\begin{equation*}
\psi \left( x\right) =\left\{ 
\begin{array}{ll}
1 & \text{for }\left\vert x\right\vert \leq L \\ 
0 & \text{for }\left\vert x\right\vert \geq 2L%
\end{array}%
\right.
\end{equation*}%
Setting $w=\psi u$ and $v=\left( 1-\psi \right) u$ where $u$ is a solution
of $\left( \ref{system}\right) $ with initial data in $H_{0}^{1}\left(
\Omega \right) \times L^{2}\left( \Omega \right) $. Let%
\begin{equation*}
X\left( t\right) =\int_{\Omega }v\left( t\right) \partial _{t}v\left(
t\right) dx+\frac{1}{2}\int_{\Omega }a\left( x\right) \left\vert v\left(
t\right) \right\vert ^{2}dx+kE_{u}\left( t\right) ,
\end{equation*}%
where $k$ is a positive constant. We have%
\begin{eqnarray}
&&X\left( t+T\right) -X\left( t\right) +\int_{t}^{t+T}E_{u}\left( s\right)
ds+\left( k-\frac{2}{\epsilon _{0}}\right) \int_{t}^{t+T}\int_{\Omega
}a\left\vert \partial _{t}u\right\vert ^{2}dxds  \notag \\
&\leq &2\int_{t}^{t+T}E_{w}\left( s\right) ds+\int_{t}^{t+T}\int_{\Omega
}\left\vert \nabla \psi \right\vert ^{2}\left\vert u\right\vert ^{2}dxds.
\label{X t estimate}
\end{eqnarray}
\end{lemma}

\begin{proof}
Noting that for each $\left( u_{0},u_{1}\right) $ in $H_{0}^{1}\left( \Omega
\right) \times L^{2}\left( \Omega \right) $ the solution $u$ of $\left( \ref%
{system}\right) $ are given as a limit of smooth solution $u_{n}$ with
initial data $\left( u_{n,0},u_{n,1}\right) $ smooth such that $\left(
u_{n,0},u_{n,1}\right) \underset{n\rightarrow +\infty }{\longrightarrow }%
\left( u_{0},u_{1}\right) $ in $H_{0}^{1}\left( \Omega \right) \times
L^{2}\left( \Omega \right) $. Note that 
\begin{equation*}
\left\Vert u_{n}\left( t,.\right) -u\left( t,.\right) \right\Vert
_{H^{1}}+\left\Vert \partial _{t}u_{n}\left( t,.\right) -\partial
_{t}u\left( t,.\right) \right\Vert _{L^{2}}\underset{n\rightarrow +\infty }{%
\longrightarrow }0,
\end{equation*}%
uniformly on the each closed interval $\left[ 0,T\right] $ for any $T>0$.
Therefore we may assume that $u$ is smooth.

We have $v=\left( 1-\psi \right) u$. Then $v$ is a solution of 
\begin{equation}
\left\{ 
\begin{array}{ll}
\partial _{t}^{2}v-\Delta v+a\left( x\right) \partial _{t}v=f\left(
t,x\right) & 
\mathbb{R}
_{+}\times \Omega , \\ 
v=0 & 
\mathbb{R}
_{+}\times \Gamma , \\ 
\left( v\left( 0\right) ,\partial _{t}v\left( 0\right) \right) =\left(
1-\psi \right) \left( u_{0},u_{1}\right) , & 
\end{array}%
\right.  \label{sys v}
\end{equation}%
with 
\begin{equation*}
f\left( t,x\right) =2\nabla \psi \nabla u+u\Delta \psi .
\end{equation*}%
Using the fact that $v$ is a solution of $\left( \ref{sys v}\right) $ and
that%
\begin{equation*}
\frac{d}{dt}E_{u}\left( t\right) =-\int_{\Omega }a\left( x\right) \left\vert
\partial _{t}u\left( t\right) \right\vert ^{2}dx,
\end{equation*}%
we deduce that%
\begin{equation*}
\frac{d}{dt}X\left( t\right) =\int_{\Omega }\left\vert \partial _{t}v\left(
t\right) \right\vert ^{2}-\left\vert \nabla v\left( t\right) \right\vert
^{2}dx-k\int_{\Omega }a\left( x\right) \left\vert \partial _{t}u\left(
t\right) \right\vert ^{2}dx+\int_{\Omega }f\left( t,x\right) vdx.
\end{equation*}%
Since%
\begin{eqnarray*}
\int_{\Omega }f\left( t,x\right) vdx &=&\int_{\Omega }\left( 2\nabla \psi
\nabla u+u\Delta \psi \right) \left( 1-\psi \right) udx \\
&=&\int_{\Omega }\nabla \psi \nabla u^{2}+u^{2}\Delta \psi dx-\int_{\Omega }%
\frac{1}{2}\nabla \psi ^{2}\nabla u^{2}+u^{2}\psi \Delta \psi dx \\
&=&\int_{\Omega }\left\vert \nabla \psi \right\vert ^{2}\left\vert
u\right\vert ^{2}dx.
\end{eqnarray*}%
Thus we obtain%
\begin{eqnarray*}
\frac{d}{dt}X\left( t\right) &=&\int_{\Omega }\left\vert \partial
_{t}v\left( t\right) \right\vert ^{2}-\left\vert \nabla v\left( t\right)
\right\vert ^{2}dx-k\int_{\Omega }a\left( x\right) \left\vert \partial
_{t}u\left( t\right) \right\vert ^{2}dx+\int_{\Omega }\left\vert \nabla \psi
\right\vert ^{2}\left\vert u\right\vert ^{2}dx \\
&=&2\int_{\Omega }\left\vert \partial _{t}v\left( t\right) \right\vert
^{2}dx-2E_{v}\left( t\right) -k\int_{\Omega }a\left( x\right) \left\vert
\partial _{t}u\left( t\right) \right\vert ^{2}dx+\int_{\Omega }\left\vert
\nabla \psi \right\vert ^{2}\left\vert u\right\vert ^{2}dx \\
&\leq &2\int_{\Omega }\left\vert \partial _{t}v\left( t\right) \right\vert
^{2}dx-E_{u}\left( t\right) +2E_{w}\left( t\right) -k\int_{\Omega }a\left(
x\right) \left\vert \partial _{t}u\left( t\right) \right\vert
^{2}dx+\int_{\Omega }\left\vert \nabla \psi \right\vert ^{2}\left\vert
u\right\vert ^{2}dx.
\end{eqnarray*}%
Using the fact that the support of $\left( 1-\psi \right) $ is contained in
the set $\left\{ x\in \Omega ,a\left( x\right) >\epsilon _{0}\right\} $, we
infer that%
\begin{equation*}
\int_{\Omega }\left\vert \partial _{t}v\left( t\right) \right\vert
^{2}dx=\int_{\Omega }\left\vert \left( 1-\psi \right) \partial
_{t}u\right\vert ^{2}dx\leq \frac{1}{\epsilon _{0}}\int_{\Omega }a\left(
x\right) \left\vert \partial _{t}u\left( t\right) \right\vert ^{2}dx.
\end{equation*}%
This gives%
\begin{equation}
\frac{d}{dt}X\left( t\right) +E_{u}\left( t\right) +\left( k-\frac{2}{%
\epsilon _{0}}\right) \int_{\Omega }a\left( x\right) \left\vert \partial
_{t}u\left( t\right) \right\vert ^{2}dx\leq 2E_{w}\left( t\right)
+\int_{\Omega }\left\vert \nabla \psi \right\vert ^{2}\left\vert
u\right\vert ^{2}dx.  \label{Xt derivetive estimate}
\end{equation}%
Integrating the estimate above between $t$ and $t+T$ we get $\left( \ref{X t
estimate}\right) $.
\end{proof}

\subsection{Proof of Theorem 1}

In the sequel $C,$ $C_{T}$ and $C_{T,\delta }$ denote a generic positive
constants and any changes from one derivation to the next will not be
explicitly outlined.

Let $\psi \in C_{0}^{\infty }\left( 
\mathbb{R}
^{d}\right) $ such that $0\leq \psi \leq 1$ and%
\begin{equation*}
\psi \left( x\right) =\left\{ 
\begin{array}{ll}
1 & \text{for }\left\vert x\right\vert \leq L \\ 
0 & \text{for }\left\vert x\right\vert \geq 2L%
\end{array}%
\right. 
\end{equation*}%
Setting $w=\psi u$ and $v=\left( 1-\psi \right) u$ where $u$ is a solution
of $\left( \ref{system}\right) $ with initial data in $H_{0}^{1}\left(
\Omega \right) \times L^{2}\left( \Omega \right) $. Let%
\begin{equation*}
X\left( t\right) =\int_{\Omega }v\left( t\right) \partial _{t}v\left(
t\right) dx+\frac{1}{2}\int_{\Omega }a\left( x\right) \left\vert v\left(
t\right) \right\vert ^{2}dx+kE_{u}\left( t\right) .
\end{equation*}%
According to lemma \ref{lemma xt},%
\begin{eqnarray*}
&&X\left( t+T\right) -X\left( t\right) +\int_{t}^{t+T}E_{u}\left( s\right)
ds+\left( k-\frac{2}{\epsilon _{0}}\right) \int_{t}^{t+T}\int_{\Omega
}a\left\vert \partial _{t}u\right\vert ^{2}dxds \\
&\leq &2\int_{t}^{t+T}E_{w}\left( s\right) ds+\int_{t}^{t+T}\int_{\Omega
}\left\vert \nabla \psi \right\vert ^{2}\left\vert u\right\vert ^{2}dxds.
\end{eqnarray*}

We have $w=\psi u$. Then $w$ is a solution of 
\begin{equation*}
\left\{ 
\begin{array}{ll}
\partial _{t}^{2}w-\Delta w+a\left( x\right) \partial _{t}w=f\left(
t,x\right)  & 
\mathbb{R}
_{+}\times M, \\ 
w=0 & 
\mathbb{R}
_{+}\times \partial M, \\ 
\left( w\left( 0\right) ,\partial _{t}w\left( 0\right) \right) =\left( \psi
u_{0},\psi u_{1}\right) , & 
\end{array}%
\right. 
\end{equation*}%
with $M=\Omega \cap B_{2L}$ and 
\begin{equation*}
f\left( t,x\right) =-2\nabla \psi \nabla u-u\Delta \psi \in
L_{loc}^{2}\left( 
\mathbb{R}
_{+},L^{2}\left( M\right) \right) .
\end{equation*}%
Since $\left( \psi u_{0},\psi u_{1}\right) \in H_{0}^{1}\left( M\right)
\times L^{2}\left( M\right) $, using $\left( \ref{observability bounded}%
\right) $ we infer that the following inequality%
\begin{equation}
E_{w}\left( t\right) \leq C_{T}^{1}\left( \int_{s}^{s+T}\int_{\Omega
}a\left\vert \partial _{t}w\right\vert ^{2}+\left\vert f\left( \tau
,x\right) \right\vert ^{2}dxd\tau \right) ,
\label{energy w observability proof thm 1}
\end{equation}%
holds for every $t\geq 0$. According to ( \cite[proposition 2]{daou})%
\begin{eqnarray*}
E_{w}\left( s\right)  &\leq &2e^{s-t}\left( E_{w}\left( t\right)
+\int_{t}^{s}\int_{\Omega }\left\vert f\left( \tau ,x\right) \right\vert
^{2}dxd\tau \right)  \\
&\leq &2e^{T}\left( E_{w}\left( t\right) +\int_{t}^{t+T}\int_{\Omega
}\left\vert f\left( \tau ,x\right) \right\vert ^{2}dxd\tau \right) ,\text{
for }t\leq s\leq t+T.
\end{eqnarray*}%
The estimate above combined with $\left( \ref{energy w observability proof
thm 1}\right) $ gives%
\begin{equation*}
\int_{t}^{t+T}E_{w}\left( s\right) ds\leq \tilde{C}_{T}\left(
\int_{t}^{t+T}\int_{\Omega }a\left\vert \partial _{t}w\right\vert
^{2}+\left\vert f\left( \tau ,x\right) \right\vert ^{2}dxd\tau \right) ,
\end{equation*}%
since%
\begin{equation*}
\int_{t}^{t+T}\int_{\Omega }a\left\vert \partial _{t}w\right\vert
^{2}dxd\tau \leq \int_{t}^{t+T}\int_{\Omega }a\left\vert \partial
_{t}u\right\vert ^{2}dxd\tau ,
\end{equation*}%
we obtain%
\begin{equation*}
\int_{t}^{t+T}E_{w}\left( s\right) ds\leq \tilde{C}_{T}\left(
\int_{t}^{t+T}\int_{\Omega }a\left\vert \partial _{t}u\right\vert
^{2}+\left\vert f\left( \tau ,x\right) \right\vert ^{2}dxd\tau \right) ,
\end{equation*}%
On the other hand, using Poincare's inequality we obtain%
\begin{equation*}
\int_{t}^{t+T}\int_{\Omega }\left\vert f\left( \tau ,x\right) \right\vert
^{2}dxd\tau \leq C_{L}\int_{t}^{t+T}\int_{\Omega \cap B_{2L}}\left\vert
\nabla u\right\vert ^{2}dxds,
\end{equation*}%
for some $C_{L}>0$. The estimate $\left( \ref{observability gradient}\right) 
$, gives%
\begin{equation*}
\int_{t}^{t+T}E_{w}\left( s\right) ds\leq C_{T,\delta }\left(
\int_{t}^{t+T}\int_{\Omega }a\left\vert \partial _{t}u\right\vert
^{2}dxd\tau \right) +C_{L}\tilde{C}_{T}\delta E_{u}\left( t\right) ,
\end{equation*}%
since%
\begin{equation}
E_{u}\left( t\right) \leq \int_{t}^{t+T}\int_{\Omega }a\left\vert \partial
_{t}u\right\vert ^{2}dxd\tau +\frac{1}{T}\int_{t}^{t+T}E_{u}\left( s\right)
ds,  \label{proof of theorem 1 Eu integral estimate}
\end{equation}%
we deduce that%
\begin{equation*}
\int_{t}^{t+T}E_{w}\left( s\right) ds\leq C_{T,\delta }\left(
\int_{t}^{t+T}\int_{\Omega }a\left\vert \partial _{t}u\right\vert
^{2}dxd\tau \right) +C_{T}\delta \int_{t}^{t+T}E_{u}\left( s\right) ds.
\end{equation*}%
Using $\left( \ref{observability gradient}\right) $ and $\left( \ref{proof
of theorem 1 Eu integral estimate}\right) $, we infer that%
\begin{equation*}
\int_{t}^{t+T}\int_{\Omega }\left\vert \nabla \psi \right\vert
^{2}\left\vert u\right\vert ^{2}dxds\leq C_{T,\delta }\left(
\int_{t}^{t+T}\int_{\Omega }a\left\vert \partial _{t}u\right\vert
^{2}dxd\tau \right) +C_{T}\delta \int_{t}^{t+T}E_{u}\left( s\right) ds.
\end{equation*}%
\ Now $\left( \ref{X t estimate}\right) $ and the estimates above gives%
\begin{equation*}
\begin{array}{c}
X\left( t+T\right) -X\left( t\right) +\left( 1-3C_{T}\delta \right)
\int_{t}^{t+T}E_{u}\left( s\right) ds+\left( k-\frac{2}{\epsilon _{0}}%
-3C_{T,\delta }\right) \int_{t}^{t+T}\int_{\Omega }a\left\vert \partial
_{t}u\right\vert ^{2}dxds\leq 0.%
\end{array}%
\end{equation*}%
We have%
\begin{eqnarray}
X\left( t\right)  &\leq &\frac{3}{2}\left\Vert a\right\Vert _{\infty
}\int_{\Omega }\left\vert v\left( t\right) \right\vert ^{2}dx+\left( k+\frac{%
2}{\epsilon _{0}}\right) E_{u}\left( t\right) \text{ and}  \notag \\
X\left( t\right)  &\geq &\frac{\epsilon _{0}}{4}\int_{\Omega }\left\vert
v\left( t\right) \right\vert ^{2}dx+\left( k-\frac{8}{\epsilon _{0}}\right)
E_{u}\left( t\right) .  \label{k estimate}
\end{eqnarray}%
We choose $\delta $ and $k$ such that 
\begin{eqnarray*}
1-3C_{T}\delta  &=&\frac{1}{2}, \\
k-\frac{2}{\epsilon _{0}}-3C_{T,\delta } &\geq &\epsilon >0\text{ and }k-%
\frac{8}{\epsilon _{0}}\geq \epsilon ,
\end{eqnarray*}%
therefore we obtain%
\begin{equation}
X\left( t+T\right) -X\left( t\right) +\frac{1}{2}\int_{t}^{t+T}E_{u}\left(
s\right) ds+\epsilon \int_{t}^{t+T}\int_{\Omega }a\left\vert \partial
_{t}u\right\vert ^{2}dxds\leq 0,  \label{Xt final estimate}
\end{equation}%
this gives%
\begin{equation*}
X\left( nT\right) +\frac{1}{2}\int_{0}^{nT}E_{u}\left( s\right) ds\leq
X\left( 0\right) ,\text{ for all }n\in 
\mathbb{N}
.
\end{equation*}%
So there exists a positive constant $C$ such that 
\begin{eqnarray}
\underset{%
\mathbb{R}
_{+}}{\sup }X\left( t\right) +\int_{0}^{+\infty }E_{u}\left( s\right) ds
&\leq &CX\left( 0\right)   \notag \\
&\leq &CI_{0},  \label{estimate integral Eu t}
\end{eqnarray}%
with%
\begin{equation*}
I_{0}=\left\Vert u_{0}\right\Vert _{H^{1}}^{2}+\left\Vert u_{1}\right\Vert
_{L^{2}}^{2}.
\end{equation*}%
\ Since $1-\psi \equiv 1$ for $\left\vert x\right\vert \geq 2L$ and 
\begin{equation*}
X\left( t\right) \geq \frac{\epsilon _{0}}{4}\int_{\Omega }\left\vert
v\left( t\right) \right\vert ^{2}dx\geq \frac{\epsilon _{0}}{4}\int_{\left\{
\left\vert x\right\vert \geq 2L\right\} }\left\vert u\left( t\right)
\right\vert ^{2}dx,
\end{equation*}%
therefore using $\left( \ref{estimate integral Eu t}\right) $\ we obtain%
\begin{equation*}
\underset{%
\mathbb{R}
_{+}}{\sup }\int_{\left\{ \left\vert x\right\vert \geq 2L\right\}
}\left\vert u\left( t\right) \right\vert ^{2}dx\leq \frac{4C}{\epsilon _{0}}%
I_{0}.
\end{equation*}%
Poincare's inequality and the fact that the energy of $u$ is decreasing gives%
\begin{equation*}
\int_{\Omega \cap B_{2L}}\left\vert u\left( t\right) \right\vert ^{2}dx\leq
C_{L}\int_{\Omega }\left\vert \nabla u\left( t\right) \right\vert ^{2}dx\leq
C_{L}E_{u}\left( 0\right) .
\end{equation*}%
Combining the last two estimates, we get 
\begin{equation*}
\underset{%
\mathbb{R}
_{+}}{\sup }\int_{\Omega }\left\vert u\left( t\right) \right\vert ^{2}dx\leq 
\frac{4C}{\epsilon _{0}}I_{0}+C_{L}E_{u}\left( 0\right) \leq CI_{0}.
\end{equation*}%
The energy decay estimate follows from $\left( \ref{estimate integral Eu t}%
\right) $\ and the fact that%
\begin{equation}
\left( 1+t\right) E_{u}\left( t\right) \leq E_{u}\left( 0\right)
+\int_{0}^{+\infty }E_{u}\left( s\right) ds\leq CI_{0},\text{ for all }t\geq
0.  \label{theorem 1 energy result}
\end{equation}

This finishes the proof of theorem 1, now we give the proof of proposition 1.

\subsection{Proof of proposition 1}

Let $n\in 
\mathbb{N}
^{\ast }$ and $u$ solution of $\left( \ref{system}\right) $ with initial
data $\left( u_{0},u_{1}\right) $\ in $D\left( A^{n}\right) $ such that $%
u_{0}\in L^{2}\left( \Omega \right) $. We set $u_{n}=\partial _{t}^{n}u$.
First we prove 
\begin{equation}
\int_{0}^{+\infty }\left( 1+s\right) ^{n}E_{u_{n}}\left( s\right) ds\leq
C_{n}I_{0,n},\text{ for all }n\in 
\mathbb{N}
\label{1+t integral decay estimate regular data}
\end{equation}%
where%
\begin{equation*}
I_{0,n}=\sum_{i=0}^{n}\left\Vert A^{i}\left( 
\begin{array}{c}
u_{0} \\ 
u_{1}%
\end{array}%
\right) \right\Vert _{H}^{2}+\left\Vert u_{0}\right\Vert _{L^{2}}^{2}
\end{equation*}

Let $u$ be a solution of $\left( \ref{system}\right) $ with initial data $%
\left( u_{0},u_{1}\right) $\ in $D\left( A^{0}\right) $ such that $u_{0}\in
L^{2}\left( \Omega \right) $. From $\left( \ref{estimate integral Eu t}%
\right) $ we infer that\ 
\begin{equation*}
\int_{0}^{+\infty }E_{u}\left( s\right) ds\leq C_{0}I_{0}.
\end{equation*}

We assume that the following estimate%
\begin{equation}
\int_{0}^{+\infty }\left( 1+s\right) ^{p}E_{u_{p}}\left( s\right) ds\leq
C_{p}I_{0,p},  \label{hypothese}
\end{equation}%
holds, for all solution $u$ of $\left( \ref{system}\right) $ with initial
data $\left( u_{0},u_{1}\right) $\ in $D\left( A^{p}\right) $ such that $%
u_{0}\in L^{2}\left( \Omega \right) $.

Let $u$ be a solution of $\left( \ref{system}\right) $ with initial data $%
\left( u_{0},u_{1}\right) $ in $D\left( A^{p+1}\right) $ such that $u_{0}\in
L^{2}\left( \Omega \right) $.

We have $u_{p+1}=\partial _{t}^{p}\left( \partial _{t}u\right) $. Since $%
\left( \partial _{t}u\left( 0\right) ,\partial _{t}^{2}u\left( 0\right)
\right) =A\left( 
\begin{array}{c}
u_{0} \\ 
u_{1}%
\end{array}%
\right) \in D\left( A^{p}\right) $ and $\partial _{t}u\left( 0\right)
=u_{1}\in L^{2}\left( \Omega \right) $. According to $\left( \ref{hypothese}%
\right) $, we have%
\begin{eqnarray}
\int_{0}^{+\infty }\left( 1+s\right) ^{p}E_{u_{p+1}}\left( s\right) ds &\leq
&C_{p}\left( \sum_{i=0}^{p}\left\Vert A^{i+1}\left( 
\begin{array}{c}
u_{0} \\ 
u_{1}%
\end{array}%
\right) \right\Vert _{H}^{2}+\left\Vert u_{1}\right\Vert _{L^{2}}^{2}\right) 
\notag \\
&\leq &C_{p}I_{0,p+1}.  \label{hypothese 1}
\end{eqnarray}

Let $\psi \in C_{0}^{\infty }\left( 
\mathbb{R}
^{d}\right) $ such that $0\leq \psi \leq 1$ and%
\begin{equation*}
\psi \left( x\right) =\left\{ 
\begin{array}{cc}
1 & \text{for }\left\vert x\right\vert \leq L \\ 
0 & \text{for }\left\vert x\right\vert \geq 2L%
\end{array}%
\right. 
\end{equation*}%
Setting $w=\psi u_{p+1}$ and $v=\left( 1-\psi \right) u_{p+1}$. Let%
\begin{equation*}
X\left( t\right) =\int_{\Omega }v\left( t\right) \partial _{t}v\left(
t\right) dx+\frac{1}{2}\int_{\Omega }a\left( x\right) \left\vert v\left(
t\right) \right\vert ^{2}dx+kE_{u_{p+1}}\left( t\right) ,
\end{equation*}%
where $k$ is a positive constant. $u_{p+1}$ satisfies 
\begin{equation*}
\left\{ 
\begin{array}{lc}
\partial _{t}^{2}u_{p+1}-\Delta u_{p+1}+a\left( x\right) \partial
_{t}u_{p+1}=0 & \text{in }\mathbb{R}_{+}\times \Omega , \\ 
u_{p+1}=0 & \text{on }\mathbb{R}_{+}\times \Gamma , \\ 
\left( u_{p+1}\left( 0,x\right) ,\partial _{t}u_{p+1}\left( 0,x\right)
\right) \in H_{0}^{1}\left( \Omega \right) \times L^{2}\left( \Omega \right)
, & 
\end{array}%
\right. 
\end{equation*}%
Then we know from $\left( \ref{Xt final estimate}\right) $ that%
\begin{equation}
X\left( t+T\right) -X\left( t\right) +\frac{1}{2}\int_{t}^{t+T}E_{u_{p+1}}%
\left( s\right) ds+\epsilon \int_{t}^{t+T}\int_{\Omega }a\left\vert \partial
_{t}u_{p+1}\right\vert ^{2}dxds\leq 0.  \label{X1 t estimate}
\end{equation}%
Multiplying the estimate above by $\left( 1+t+T\right) ^{p+1}$, we obtain%
\begin{equation*}
\begin{array}{l}
\left( 1+t+T\right) ^{p+1}X\left( t+T\right) -\left( 1+t\right)
^{p+1}X\left( t\right)  \\ 
+\frac{1}{2}\int_{t}^{t+T}\left( 1+s\right) ^{p+1}E_{u_{p+1}}\left( s\right)
ds\leq C_{T}\left( 1+t\right) ^{p}X\left( t\right) .%
\end{array}%
\end{equation*}%
Therefore using $\left( \ref{hypothese}\right) ,$ $\left( \ref{hypothese 1}%
\right) $\ and the fact that%
\begin{equation}
X\left( t\right) \leq \frac{3}{2}\left\Vert a\right\Vert _{\infty
}E_{u_{p}}\left( t\right) +\left( k+\frac{2}{\epsilon _{0}}\right)
E_{u_{p+1}}\left( t\right) ,  \notag
\end{equation}%
we deduce that for any $q\in 
\mathbb{N}
^{\ast }$%
\begin{eqnarray*}
\frac{1}{2}\int_{0}^{qT}\left( 1+s\right) ^{p+1}E_{u_{p+1}}\left( s\right)
ds &\leq &C_{T}\overset{q-1}{\underset{i=0}{\sum }}\left( 1+iT\right)
^{p}X\left( iT\right)  \\
&\leq &C_{T}\overset{q-1}{\underset{i=0}{\sum }}\left( 1+iT\right)
^{p}\left( E_{u_{p+1}}\left( iT\right) +E_{u_{p}}\left( iT\right) \right)  \\
&\leq &C_{T}\overset{q-1}{\underset{i=0}{\sum }}\int_{iT}^{\left( i+1\right)
T}\left( 1+s\right) ^{p}\left( E_{u_{p+1}}\left( s\right) +E_{u_{p}}\left(
s\right) \right) ds \\
&&+C_{T}\left( E_{u_{p+1}}\left( 0\right) +E_{u_{p}}\left( 0\right) \right) 
\\
&\leq &C_{T}\int_{0}^{+\infty }\left( 1+s\right) ^{p}\left(
E_{u_{p+1}}\left( s\right) +E_{u_{p}}\left( s\right) \right) ds \\
&&+C_{T}\left( E_{u_{p+1}}\left( 0\right) +E_{u_{p}}\left( 0\right) \right) 
\\
&\leq &C_{T,p}I_{0,p+1}.
\end{eqnarray*}%
We deduce that%
\begin{equation*}
\int_{0}^{+\infty }\left( 1+s\right) ^{p+1}E_{u_{p+1}}\left( s\right) ds\leq
C_{p+1}I_{0,p+1}
\end{equation*}%
We remind that we have proved that 
\begin{equation*}
\int_{0}^{+\infty }\left( 1+s\right) ^{n}E_{u_{n}}\left( s\right) ds\leq
C_{n}I_{0,n}.
\end{equation*}%
Now the energy decay estimate follows from the fact that%
\begin{eqnarray*}
\left( 1+t\right) ^{n+1}E_{u_{n}}\left( t\right)  &\leq &E_{u_{n}}\left(
0\right) +\left( n+1\right) \int_{0}^{t}\left( 1+s\right)
^{n}E_{u_{n}}\left( s\right) ds \\
&\leq &C_{n}I_{0,n},
\end{eqnarray*}%
for all $t\geq 0$. Now using the estimate above, we infer that,%
\begin{eqnarray*}
\left( 1+t\right) ^{n}\left\Vert \partial _{t}^{n}u\left( t\right)
\right\Vert _{L^{2}}^{2} &\leq &2\left( 1+t\right) ^{n}E_{u_{n-1}}\left(
t\right)  \\
&\leq &C_{1}I_{0,n-1}\text{ for all }t\geq 0,\text{ }
\end{eqnarray*}

We have $\partial _{t}^{n-1}u$ is a solution of the following system 
\begin{equation}
\left\{ 
\begin{array}{lc}
\partial _{t}^{n+1}u-\Delta \partial _{t}^{n-1}u+a\left( x\right) \partial
_{t}^{n}u=0 & \text{in }\mathbb{R}_{+}\times \Omega , \\ 
\partial _{t}^{n-1}u=0 & \text{on }\mathbb{R}_{+}\times \Gamma , \\ 
\left( \partial _{t}^{n-1}u\left( 0,x\right) ,\partial _{t}^{n}u\left(
0,x\right) \right) \in D\left( A\right)  & 
\end{array}%
\right.   \label{sys dtnu}
\end{equation}%
therefore%
\begin{equation*}
\left( \partial _{t}^{n-1}u,\partial _{t}^{n}u\right) \in C^{1}\left( 
\mathbb{R}
_{+},H\right) .
\end{equation*}%
Using Eq $\left( \ref{sys dtnu}\right) ,$ we infer that%
\begin{eqnarray*}
\left( 1+t\right) ^{n}\left\Vert \Delta \partial _{t}^{n-1}u\left( t\right)
\right\Vert _{L^{2}}^{2} &\leq &C\left( 1+t\right) ^{n}\left(
E_{u_{n}}\left( t\right) +E_{u_{n-1}}\left( t\right) \right)  \\
&\leq &CI_{0,n}.
\end{eqnarray*}

\section{Proof of Theorem 2}

This section is devoted to the proof of theorem 2 and proposition 2, we
begin by giving some preliminary results.

\begin{proposition}
\label{proposition onservability global copy(1)}We assume that Hyp A holds
and $(\omega $,$T)$ geometrically controls $\Omega $. Let $\delta >0$ and $%
\chi \in C_{0}^{\infty }\left( 
\mathbb{R}
^{d}\right) $. There exists $C_{T,\delta ,\chi }>0$, such that the following
inequality%
\begin{equation}
\begin{array}{c}
\int_{t}^{t+T}\int_{\Omega }\chi ^{2}\left( x\right) \left( 1+s\right)
\left\vert \nabla u\right\vert ^{2}dxds\leq C_{T,\delta ,\chi }\left(
\int_{t}^{t+T}\int_{\Omega }a\left( x\right) \left( 1+s\right) \left\vert
\partial _{t}u\right\vert ^{2}dxds\right) \\ 
\text{ \ \ \ \ \ \ \ \ \ \ \ \ \ \ \ \ \ \ \ \ \ \ \ \ \ \ \ \ \ \ \ \ \ \ \
\ \ \ \ \ \ \ \ \ \ \ \ \ \ \ \ \ }+\delta \left( \left( 1+t\right)
E_{u}\left( t\right) +\int_{t}^{t+T}\int_{\Omega }\left\vert u\right\vert
^{2}dxds\right) .%
\end{array}
\label{observability gradient 1+s}
\end{equation}%
holds for every $t\geq 0$ and for all $u$ solution of $\left( \ref{system}%
\right) $ with initial data $\left( u_{0},u_{1}\right) $ in $H_{0}^{1}\times
L^{2}$.
\end{proposition}

\begin{proof}
To prove this result we argue by contradiction: If $\left( \ref%
{observability gradient 1+s}\right) $ was false, there would exist a
sequence of numbers $\left( t_{n}\right) $ and a sequence of solutions $%
\left( u_{n}\right) $ such that%
\begin{equation*}
\begin{array}{c}
\int_{t_{n}}^{t_{n}+T}\int_{\Omega }\chi ^{2}\left( x\right) \left(
1+s\right) \left\vert \nabla u_{n}\right\vert ^{2}dxds\geq n\left(
\int_{t_{n}}^{t_{n}+T}\int_{\Omega }a\left( x\right) \left( 1+s\right)
\left\vert \partial _{t}u_{n}\right\vert ^{2}dxds\right) \\ 
\text{ \ \ \ \ \ \ \ \ \ \ \ \ \ \ \ \ \ \ \ \ \ \ \ \ \ \ \ \ \ \ \ \ \ \ \
\ \ \ \ \ }+\delta \left( \left( 1+t_{n}\right) E_{u_{n}}\left( t_{n}\right)
+\int_{t_{n}}^{t_{n}+T}\int_{\Omega }\left\vert u_{n}\right\vert
^{2}dxds\right) .%
\end{array}%
\end{equation*}%
We may assume that $t_{n}\underset{n\rightarrow +\infty }{\longrightarrow }%
+\infty $ (if the sequence $t_{n}$ is bounded we can argue as in the proof
of proposition \ref{proposition onservability global}). Setting%
\begin{equation*}
\lambda _{n}^{2}=\int_{t_{n}}^{t_{n}+T}\int_{\Omega }\chi ^{2}\left(
x\right) \left( 1+s\right) \left\vert \nabla u_{n}\right\vert ^{2}dxds\text{
and }v_{n}=\frac{\left( 1+t_{n}+\cdot \right) ^{1/2}u_{n}\left( t_{n}+\cdot
\right) }{\lambda _{n}}.
\end{equation*}%
We have%
\begin{equation}
\frac{1}{\lambda _{n}^{2}}\left( 1+t_{n}\right) E_{u_{n}}\left( t_{n}\right)
\leq \frac{1}{\delta }\text{ and }\frac{1}{\lambda _{n}^{2}}%
\int_{0}^{T}\int_{\Omega }\left\vert u_{n}\left( t_{n}+s\right) \right\vert
^{2}dxds\leq \frac{1}{\delta },  \label{proof of proposition  contradiction}
\end{equation}%
moreover%
\begin{equation}
\frac{1}{\lambda _{n}^{2}}\int_{t_{n}}^{t_{n}+T}\int_{\Omega }a\left(
x\right) \left( 1+s\right) \left\vert \partial _{t}u_{n}\right\vert ^{2}dxds%
\underset{n\rightarrow +\infty }{\longrightarrow }0.
\label{proof proposition 5 1}
\end{equation}%
$v_{n}$ is a solution of the following system%
\begin{equation*}
\left\{ 
\begin{array}{lc}
\partial _{t}^{2}v_{n}-\Delta v_{n}+a\partial _{t}v_{n}=f_{n}\left(
t,x\right) & \text{in }%
\mathbb{R}
_{+}\times \Omega , \\ 
v_{n}\left( t,x\right) =0 & \text{on }%
\mathbb{R}
_{+}\times \Gamma , \\ 
\left( v_{n,0},v_{n,1}\right) \in H_{0}^{1}\left( \Omega \right) \times
L^{2}\left( \Omega \right) , & 
\end{array}%
\right.
\end{equation*}%
with%
\begin{eqnarray*}
f_{n}\left( t,x\right) &=&\frac{1}{2\lambda _{n}}\left( 1+t_{n}+t\right) ^{-%
\frac{1}{2}}\left( a\left( x\right) -\frac{1}{2}\left( 1+t_{n}+t\right)
^{-1}\right) u_{n}\left( t_{n}+t\right) \\
&&+\frac{1}{\lambda _{n}}\left( 1+t_{n}+t\right) ^{-\frac{1}{2}}\partial
_{t}u_{n}\left( t_{n}+t\right) .
\end{eqnarray*}%
It is clear that $\left( \ref{proof of proposition contradiction}\right) $,
gives%
\begin{eqnarray*}
&&\int_{0}^{T}\int_{\Omega }\left\vert \frac{1}{2\lambda _{n}}\left(
1+t_{n}+t\right) ^{-\frac{1}{2}}\left( a\left( x\right) -\frac{1}{2}\left(
1+t_{n}+t\right) ^{-1}\right) u_{n}\left( t_{n}+t\right) \right\vert ^{2}dxdt
\\
&\leq &C\left( 1+t_{n}\right) ^{-1}\frac{1}{\lambda _{n}^{2}}%
\int_{t_{n}}^{t_{n}+T}\int_{\Omega }\left\vert u_{n}\left( s\right)
\right\vert ^{2}dxds \\
&\leq &C\frac{\left( 1+t_{n}\right) ^{-1}}{\delta }
\end{eqnarray*}%
and%
\begin{eqnarray*}
\int_{0}^{T}\int_{\Omega }\left\vert \frac{1}{\lambda _{n}}\left(
1+t_{n}+t\right) ^{-\frac{1}{2}}\partial _{t}u_{n}\left( t_{n}+t\right)
\right\vert ^{2}dxdt &\leq &\frac{\left( 1+t_{n}\right) ^{-1}}{\lambda
_{n}^{2}}\int_{0}^{T}\int_{\Omega }\left\vert \partial _{t}u_{n}\left(
t_{n}+t\right) \right\vert ^{2}dxdt \\
&\leq &\frac{\left( 1+t_{n}\right) ^{-1}}{\lambda _{n}^{2}}2TE_{u_{n}}\left(
t_{n}\right) \\
&\leq &\frac{2T\left( 1+t_{n}\right) ^{-2}}{\delta }.
\end{eqnarray*}%
We conclude that%
\begin{equation}
\int_{0}^{T}\int_{\Omega }\left\vert f_{n}\left( s,x\right) \right\vert
^{2}dxds\leq C\frac{\left( 1+t_{n}\right) ^{-1}}{\delta }\underset{%
n\rightarrow +\infty }{\longrightarrow }0.
\label{proof of proposition 5 fn converge}
\end{equation}%
$\left( \ref{proof of proposition contradiction}\right) $ gives%
\begin{eqnarray*}
\int_{0}^{T}\int_{\Omega }a\left( x\right) \left\vert \partial
_{t}v_{n}\right\vert ^{2}dxdt &\leq &\left\Vert a\right\Vert _{\infty
}\left( 1+t_{n}\right) ^{-1}\frac{1}{\lambda _{n}^{2}}\int_{t_{n}}^{t_{n}+T}%
\int_{\Omega }\left\vert u_{n}\left( s\right) \right\vert ^{2}dxds \\
&&+\frac{2}{\lambda _{n}^{2}}\int_{t_{n}}^{t_{n}+T}\int_{\Omega }a\left(
x\right) \left( 1+s\right) \left\vert \partial _{t}u_{n}\right\vert ^{2}dxds
\\
&\leq &\frac{\left\Vert a\right\Vert _{\infty }\left( 1+t_{n}\right) ^{-1}}{%
\delta }+\frac{2}{\lambda _{n}^{2}}\int_{t_{n}}^{t_{n}+T}\int_{\Omega
}a\left( x\right) \left( 1+s\right) \left\vert \partial _{t}u_{n}\right\vert
^{2}dxds.
\end{eqnarray*}%
Now using $\left( \ref{proof proposition 5 1}\right) ,$ we deduce that%
\begin{equation}
\int_{0}^{T}\int_{\Omega }a\left( x\right) \left\vert \partial
_{t}v_{n}\right\vert ^{2}dxdt\underset{n\rightarrow +\infty }{%
\longrightarrow }0.  \label{proof of proposition 5 adtun converge}
\end{equation}%
We multiply the equation satisfied by $u_{n}$ by $\left( 1+t\right) \partial
_{t}u_{n}$\ and integrating between $t_{n}$ and $t_{n}+t,$ we obtain%
\begin{equation*}
\left( 1+t_{n}+t\right) E_{u_{n}}\left( t_{n}+t\right) -\left(
1+t_{n}\right) E_{u_{n}}\left( t_{n}\right)
=\int_{t_{n}}^{t_{n}+t}E_{u_{n}}\left( s\right)
ds-\int_{t_{n}}^{t_{n}+T}\int_{\Omega }a\left( x\right) \left( 1+s\right)
\left\vert \partial _{t}u_{n}\right\vert ^{2}dxds
\end{equation*}%
thus by using $\left( \ref{proof of proposition contradiction}\right) ,$ we
infer that%
\begin{eqnarray*}
\frac{1}{\lambda _{n}^{2}}\left( 1+t_{n}+t\right) E_{u_{n}}\left(
t_{n}+t\right) &\leq &\frac{1}{\lambda _{n}^{2}}\left( 1+t_{n}\right)
E_{u_{n}}\left( t_{n}\right) +\frac{1}{\lambda _{n}^{2}}%
\int_{t_{n}}^{t_{n}+t}E_{u_{n}}\left( s\right) ds \\
&\leq &\frac{1}{\delta }+\frac{T}{\delta }\text{, for all }t\in \left[ 0,T%
\right] .
\end{eqnarray*}%
\ On the other hand, we have%
\begin{eqnarray*}
\int_{0}^{T}E_{v_{n}}\left( t\right) dt &\leq &\frac{1}{\lambda _{n}^{2}}%
\int_{0}^{T}\left[ \left( 1+t_{n}+t\right) E_{u_{n}}\left( t_{n}+t\right)
+\left( 1+t_{n}+t\right) ^{-1}\int_{\Omega }\left\vert u_{n}\left(
t_{n}+t\right) \right\vert ^{2}dx\right] dt \\
&\leq &\frac{C_{T}}{\delta }.
\end{eqnarray*}%
Let $Z_{n}$ be the solution of the following system%
\begin{equation*}
\left\{ 
\begin{array}{ll}
\partial _{t}^{2}Z_{n}-\Delta Z_{n}=0 & 
\mathbb{R}
_{+}\times \Omega , \\ 
Z_{n}=0 & 
\mathbb{R}
_{+}\times \Gamma , \\ 
\left( Z_{n}\left( 0\right) ,\partial _{t}Z_{n}\left( 0\right) \right)
=\left( v_{n}\left( 0\right) ,\partial _{t}v_{n}\left( 0\right) \right) . & 
\end{array}%
\right.
\end{equation*}%
The hyperbolic energy inequality gives%
\begin{equation*}
\underset{\left[ 0,T\right] }{\sup }E_{v_{n}-Z_{n}}\left( s\right) \leq
C_{T}\left\Vert a\left( x\right) \partial _{t}v_{n}+f_{n}\left( t,x\right)
\right\Vert _{L^{2}\left( \left[ 0,T\right] ,L^{2}\left( \Omega \right)
\right) }^{2}
\end{equation*}%
Now using $\left( \ref{proof of proposition 5 adtun converge}\right) $ and $%
\left( \ref{proof of proposition 5 fn converge}\right) ,$ we deduce that%
\begin{equation*}
\underset{\left[ 0,T\right] }{\sup }E_{v_{n}-Z_{n}}\left( s\right) \underset{%
n\rightarrow +\infty }{\longrightarrow }0.
\end{equation*}%
\ Using the estimate above, we obtain%
\begin{eqnarray*}
TE_{v_{n}}\left( 0\right) &=&TE_{Z_{n}}\left( 0\right)
=\int_{0}^{T}E_{Z_{n}}\left( t\right) dt \\
&\leq &2T\underset{\left[ 0,T\right] }{\sup }E_{v_{n}-Z_{n}}\left( s\right)
+2\int_{0}^{T}E_{v_{n}}\left( t\right) dt \\
&\leq &C_{T,\delta }
\end{eqnarray*}%
this gives%
\begin{equation*}
\underset{\left[ 0,T\right] }{\sup }E_{Z_{n}}\left( s\right)
=E_{v_{n}}\left( 0\right) \leq C_{T,\delta }.
\end{equation*}%
Using the result above and ( \cite[proposition 2]{daou}), we infer that%
\begin{eqnarray*}
E_{v_{n}}\left( t\right) &\leq &2e^{T}\left( E_{v_{n}}\left( 0\right)
+\int_{0}^{T}\int_{\Omega }\left\vert f_{n}\left( \tau ,x\right) \right\vert
^{2}dxd\tau \right) \text{ for }0\leq t\leq T \\
&\leq &C_{T,\delta },\text{ for all }0\leq t\leq T.
\end{eqnarray*}%
To complete the proof we have only to argue as in the proof of the
proposition \ref{proposition onservability global}.
\end{proof}

As a corollary we have

\begin{corollary}
We assume that Hyp A holds and $(\omega $,$T)$ geometrically controls $%
\Omega $. Let $\delta $,$R>0$. There exists $C_{T,\delta ,R}>0$, such that
the following inequality%
\begin{equation}
\begin{array}{c}
\int_{t}^{t+T}\int_{\Omega \cap B_{R}}\left( 1+s\right) \left\vert \nabla
u\right\vert ^{2}dxds\leq C_{T,\delta ,R}\left( \int_{t}^{t+T}\int_{\Omega
}a\left( x\right) \left( 1+s\right) \left\vert \partial _{t}u\right\vert
^{2}dxds\right) \\ 
\text{ \ \ \ \ \ \ \ \ \ \ \ \ \ \ \ \ \ \ \ \ \ \ \ \ \ \ \ \ \ \ \ \ \ \ \
\ \ \ \ \ \ \ \ \ }+\delta \left( \left( 1+t\right) E_{u}\left( t\right)
+\int_{t}^{t+T}\int_{\Omega }\left\vert u\right\vert ^{2}dxds\right) .%
\end{array}
\label{observability gradient 1+s 11}
\end{equation}%
holds for every $t\geq 0$ and for all $u$ solution of $\left( \ref{system}%
\right) $ with initial data $\left( u_{0},u_{1}\right) $ in $H$.
\end{corollary}

In the sequel we need the following result

\begin{lemma}
Let $\psi \in C_{0}^{\infty }\left( 
\mathbb{R}
^{d}\right) $ such that $0\leq \psi \leq 1$ and%
\begin{equation*}
\psi \left( x\right) =\left\{ 
\begin{array}{ll}
1 & \text{for }\left\vert x\right\vert \leq L \\ 
0 & \text{for }\left\vert x\right\vert \geq 2L%
\end{array}%
\right.
\end{equation*}%
Setting $w=\psi u$ and $v=\left( 1-\psi \right) u$ where $u$ is a solution
of $\left( \ref{system}\right) $ with initial data in $H_{0}^{1}\left(
\Omega \right) \times L^{2}\left( \Omega \right) $. Let%
\begin{equation*}
X\left( t\right) =\int_{\Omega }v\left( t\right) \partial _{t}v\left(
t\right) dx+\frac{1}{2}\int_{\Omega }a\left( x\right) \left\vert v\left(
t\right) \right\vert ^{2}dx+kE_{u}\left( t\right) .
\end{equation*}%
where $k$ is a positive constant. We have%
\begin{equation}
\begin{array}{l}
\left( 1+t+T\right) X\left( t+T\right) -\left( 1+t\right) X\left( t\right)
+\int_{t}^{t+T}\left( 1+s\right) E_{u}\left( s\right) ds \\ 
+\left( k-\frac{2}{\epsilon _{0}}\right) \int_{t}^{t+T}\int_{\Omega }a\left(
1+s\right) \left\vert \partial _{t}u\right\vert ^{2}dxds \\ 
\leq 2\int_{t}^{t+T}\left( 1+s\right) E_{w}\left( s\right)
ds+\int_{t}^{t+T}\int_{\Omega }\left( 1+s\right) \left\vert \nabla \psi
\right\vert ^{2}\left\vert u\right\vert ^{2}dxds+\int_{t}^{t+T}X\left(
s\right) ds.%
\end{array}
\label{1+t Xt estimate}
\end{equation}
\end{lemma}

\begin{proof}
We may assume that $u$ is smooth. According to $\left( \ref{Xt derivetive
estimate}\right) $ 
\begin{equation*}
\frac{d}{dt}X\left( t\right) +E_{u}\left( t\right) +\left( k-\frac{2}{%
\epsilon _{0}}\right) \int_{\Omega }a\left( x\right) \left\vert \partial
_{t}u\left( t\right) \right\vert ^{2}dx\leq 2E_{w}\left( t\right)
+\int_{\Omega }\left\vert \nabla \psi \right\vert ^{2}\left\vert
u\right\vert ^{2}dx.
\end{equation*}%
We multiply the estimate above by $1+t$, we obtain%
\begin{eqnarray*}
&&\frac{d}{dt}\left( 1+t\right) X\left( t\right) +\left( 1+t\right)
E_{u}\left( t\right) +\left( k-\frac{2}{\epsilon _{0}}\right) \int_{\Omega
}a\left( x\right) \left( 1+t\right) \left\vert \partial _{t}u\left( t\right)
\right\vert ^{2}dx \\
&\leq &2\left( 1+t\right) E_{w}\left( t\right) +\left( 1+t\right)
\int_{\Omega }\left\vert \nabla \psi \right\vert ^{2}\left\vert u\right\vert
^{2}dx+X\left( t\right) .
\end{eqnarray*}%
We integrate the inequality above between $t$ and $t+T$, we obtain $\left( %
\ref{1+t Xt estimate}\right) $.
\end{proof}

\subsection{Proof of Theorem 2}

In the sequel $C,$ $C_{T}$ and $C_{T,\delta }$ denote a generic positive
constants and any changes from one derivation to the next will not be
explicitly outlined.

Let $\psi \in C_{0}^{\infty }\left( 
\mathbb{R}
^{d}\right) $ such that $0\leq \psi \leq 1$ and%
\begin{equation*}
\psi \left( x\right) =\left\{ 
\begin{array}{ll}
1 & \text{for }\left\vert x\right\vert \leq L \\ 
0 & \text{for }\left\vert x\right\vert \geq 2L%
\end{array}%
\right.
\end{equation*}%
Setting $w=\psi u$ and $v=\left( 1-\psi \right) u$ where $u$ is a solution
of $\left( \ref{system}\right) $ with initial data in $H_{0}^{1}\left(
\Omega \right) \times L^{2}\left( \Omega \right) $ such that%
\begin{equation*}
\left\Vert d\left( \cdot \right) \left( u_{1}+au_{0}\right) \right\Vert
_{L^{2}}<+\infty .
\end{equation*}
Let%
\begin{equation*}
X\left( t\right) =\int_{\Omega }v\left( t\right) \partial _{t}v\left(
t\right) dx+\frac{1}{2}\int_{\Omega }a\left( x\right) \left\vert v\left(
t\right) \right\vert ^{2}dx+kE_{u}\left( t\right) ,
\end{equation*}%
where $k$ is a positive constant. According to $\left( \ref{1+t Xt estimate}%
\right) $ we have%
\begin{eqnarray*}
&&\left( 1+t+T\right) X\left( t+T\right) -\left( 1+t\right) X\left( t\right)
+\int_{t}^{t+T}\left( 1+s\right) E_{u}\left( s\right) ds \\
&&+\left( k-\frac{2}{\epsilon _{0}}\right) \int_{t}^{t+T}\int_{\Omega
}a\left( 1+s\right) \left\vert \partial _{t}u\right\vert ^{2}dxds \\
&\leq &2\int_{t}^{t+T}\left( 1+s\right) E_{w}\left( s\right)
ds+\int_{t}^{t+T}\int_{\Omega }\left( 1+s\right) \left\vert \nabla \psi
\right\vert ^{2}\left\vert u\right\vert ^{2}dxds+\int_{t}^{t+T}X\left(
s\right) ds.
\end{eqnarray*}%
Let $w_{1}=\left( 1+t\right) ^{\frac{1}{2}}w$. Then $w_{1}$ is a solution of 
\begin{equation*}
\left\{ 
\begin{array}{ll}
\partial _{t}^{2}w_{1}-\Delta w_{1}+a\left( x\right) \partial _{t}w_{1}=%
\frac{1}{2}\left( 1+t\right) ^{-\frac{1}{2}}\left( a\left( x\right) -\frac{1%
}{2}\left( 1+t\right) ^{-1}\right) w+f\left( t,x\right) & 
\mathbb{R}
_{+}\times M, \\ 
w_{1}=0 & 
\mathbb{R}
_{+}\times \partial M, \\ 
\left( w_{1}\left( 0\right) ,\partial _{t}w_{1}\left( 0\right) \right)
=\left( \psi u_{0},\psi u_{1}\right) , & 
\end{array}%
\right.
\end{equation*}%
with $M=\Omega \cap B_{2L}$ and 
\begin{equation*}
f\left( t,x\right) =\left( 1+t\right) ^{-\frac{1}{2}}\partial _{t}w-\left(
1+t\right) ^{\frac{1}{2}}\left( 2\nabla \psi \nabla u+u\Delta \psi \right) .
\end{equation*}%
We have%
\begin{equation*}
\partial _{t}w_{1}=\frac{1}{2}\left( 1+t\right) ^{-\frac{1}{2}}w+\left(
1+t\right) ^{\frac{1}{2}}\partial _{t}w
\end{equation*}%
therefore, using Poincare's inequality we conclude that there exists a
positive constant $c$ such that%
\begin{equation}
\left( 1+t\right) E_{w}\left( t\right) \leq cE_{w_{1}}\left( t\right) \text{%
, for all }t\geq 0.  \label{proof of theo Ew bound}
\end{equation}%
Moreover we have%
\begin{equation*}
\int_{t}^{t+T}\int_{\Omega }a\left\vert \partial _{t}w_{1}\right\vert
^{2}dxds\leq \left( 1+\left\Vert a\right\Vert _{\infty }\right)
\int_{t}^{t+T}\int_{\Omega }a\left( 1+s\right) \left\vert \partial
_{t}w\right\vert ^{2}+\left\vert w\right\vert ^{2}dxds.
\end{equation*}%
$\left( \psi u_{0},\psi u_{1}\right) \in H_{0}^{1}\left( M\right) \times
L^{2}\left( M\right) $ and 
\begin{equation*}
\frac{1}{2}\left( 1+t\right) ^{-\frac{1}{2}}\left( a\left( x\right) -\frac{1%
}{2}\left( 1+t\right) ^{-1}\right) w+f\left( t,x\right) \in
L_{loc}^{2}\left( 
\mathbb{R}
_{+},L^{2}\left( M\right) \right) .
\end{equation*}%
Therefore using $\left( \ref{observability bounded}\right) $, we infer that%
\begin{eqnarray}
E_{w_{1}}\left( t\right) &\leq &C_{T}^{1}\left( \int_{t}^{t+T}\int_{\Omega
}a\left( 1+s\right) \left\vert \partial _{t}w_{1}\right\vert ^{2}+\left\vert
f\left( s,x\right) \right\vert ^{2}dxds\right)  \notag \\
&\leq &C_{T}^{1}\left( \int_{t}^{t+T}\int_{\Omega }a\left( 1+s\right)
\left\vert \partial _{t}w\right\vert ^{2}+\left\vert f\left( s,x\right)
\right\vert ^{2}+\left\vert w\right\vert ^{2}dxds\right) ,
\label{proof of theo Ew1 bound}
\end{eqnarray}%
holds for every $t\geq 0$. We have (cf, \cite[proposition 2]{daou})%
\begin{eqnarray*}
E_{w_{1}}\left( s\right) &\leq &4e^{s-t}\left( E_{w_{1}}\left( t\right)
+\int_{t}^{s}\int_{\Omega }\left\vert f\left( s,x\right) \right\vert
^{2}+\left\vert w\right\vert ^{2}dxds\right) \\
&\leq &4e^{T}\left( E_{w_{1}}\left( t\right) +\int_{t}^{t+T}\int_{\Omega
}\left\vert f\left( s,x\right) \right\vert ^{2}+\left\vert w\right\vert
^{2}dxds\right) ,\text{ for }t\leq s\leq t+T.
\end{eqnarray*}%
Using the estimate above, $\left( \ref{proof of theo Ew bound}\right) $ and $%
\left( \ref{proof of theo Ew1 bound}\right) ,$ we get%
\begin{equation*}
\int_{t}^{t+T}\left( 1+s\right) E_{w}\left( s\right) ds\leq \tilde{C}%
_{T}\left( \int_{t}^{t+T}\int_{\Omega }a\left( 1+s\right) \left\vert
\partial _{t}w\right\vert ^{2}+\left\vert f\left( s,x\right) \right\vert
^{2}+\left\vert w\right\vert ^{2}dxds\right) .
\end{equation*}%
Which yields to%
\begin{equation}
\int_{t}^{t+T}\left( 1+s\right) E_{w}\left( s\right) ds\leq \tilde{C}%
_{T}\left( \int_{t}^{t+T}\int_{\Omega }a\left( 1+s\right) \left\vert
\partial _{t}u\right\vert ^{2}+\left\vert f\left( s,x\right) \right\vert
^{2}+\left\vert w\right\vert ^{2}dxds\right) .
\label{proof of theorem 2 estimate1}
\end{equation}%
Now we estimate the second and the third term of the RHS of the estimate
above. For second term it is clear that%
\begin{equation}
\int_{t}^{t+T}\int_{\Omega }\left\vert f\left( s,x\right) \right\vert
^{2}dxds\leq C_{L}\int_{t}^{t+T}\int_{\Omega \cap B_{2L}}\left( 1+s\right)
\left\vert \nabla u\right\vert ^{2}+\left( 1+s\right) ^{-1}\left\vert
\partial _{t}u\right\vert ^{2}dxds,  \label{proof of theorem 2 f bound}
\end{equation}%
for some $C_{L}>0$. Using $\left( \ref{observability gradient 1+s 11}\right) 
$, we infer that%
\begin{eqnarray*}
\int_{t}^{t+T}\int_{\Omega \cap B_{2L}}\left( 1+s\right) \left\vert \nabla
u\right\vert ^{2}dxds &\leq &C_{T,\delta }\left( \int_{t}^{t+T}\int_{\Omega
}a\left( 1+s\right) \left\vert \partial _{t}u\right\vert ^{2}dxds\right) \\
&&+C_{T}\delta \left( \int_{t}^{t+T}\int_{\Omega }\left\vert u\right\vert
^{2}dxds+\left( 1+t\right) E_{u}\left( t\right) \right) .
\end{eqnarray*}%
For the second term of the RHS of $\left( \ref{proof of theorem 2 f bound}%
\right) ,$ we use $\left( \ref{observability gradient}\right) $ 
\begin{equation*}
\begin{array}{l}
\int_{t}^{t+T}\int_{\Omega \cap B_{2L}}\left( 1+s\right) ^{-1}\left\vert
\partial _{t}u\right\vert ^{2}dxds \\ 
\leq C_{T,\delta }\left( \int_{t}^{t+T}\int_{\Omega }a\left( x\right)
\left\vert \partial _{t}u\right\vert ^{2}dxds\right) +\delta E_{u}\left(
t\right) \\ 
\leq C_{T,\delta }\left( \int_{t}^{t+T}\int_{\Omega }a\left( 1+s\right)
\left\vert \partial _{t}u\right\vert ^{2}dxds\right) +C_{T}\delta \left(
\int_{t}^{t+T}\int_{\Omega }\left\vert u\right\vert ^{2}dxds+\left(
1+t\right) E_{u}\left( t\right) \right) ,%
\end{array}%
\end{equation*}%
since 
\begin{equation*}
\left( 1+t\right) E_{u}\left( t\right) \leq \int_{t}^{t+T}\int_{\Omega
}a\left( 1+s\right) \left\vert \partial _{t}u\right\vert ^{2}dxds+\frac{1}{T}%
\int_{t}^{t+T}\left( 1+s\right) E_{u}\left( s\right) ds,
\end{equation*}%
then%
\begin{equation}
\begin{array}{c}
\int_{t}^{t+T}\int_{\Omega }\left\vert f\left( s,x\right) \right\vert
^{2}dxds\leq C_{T,\delta }\left( \int_{t}^{t+T}\int_{\Omega }a\left(
1+s\right) \left\vert \partial _{t}u\right\vert ^{2}dxds\right) \\ 
\text{ \ \ \ \ \ \ \ \ \ \ \ \ \ \ \ \ }+C_{T}\delta \left(
\int_{t}^{t+T}\int_{\Omega }\left\vert u\right\vert
^{2}dxds+\int_{t}^{t+T}\left( 1+s\right) E_{u}\left( s\right) ds\right) .%
\end{array}
\label{estimate 1+T f}
\end{equation}%
To estimate the third term of the RHS of $\left( \ref{proof of theorem 2
estimate1}\right) $ we use $\left( \ref{observability gradient}\right) $ and
we obtain\ 
\begin{equation}
\begin{array}{l}
\int_{t}^{t+T}\int_{\Omega }\left\vert w\right\vert ^{2}dxds \\ 
\leq C_{T,\delta }\left( \int_{t}^{t+T}\int_{\Omega }a\left( x\right)
\left\vert \partial _{t}u\right\vert ^{2}dxds\right) +\delta E_{u}\left(
t\right) \\ 
\leq C_{T,\delta }\left( \int_{t}^{t+T}\int_{\Omega }a\left( 1+s\right)
\left\vert \partial _{t}u\right\vert ^{2}dxds\right) +C_{T}\delta \left(
\int_{t}^{t+T}\int_{\Omega }\left\vert u\right\vert
^{2}dxds+\int_{t}^{t+T}\left( 1+s\right) E_{u}\left( s\right) ds\right) .%
\end{array}
\label{estimate 1+T u}
\end{equation}%
Combining $\left( \ref{estimate 1+T f}\right) $ and $\left( \ref{estimate
1+T u}\right) $, we get%
\begin{equation*}
\begin{array}{l}
\int_{t}^{t+T}\left( 1+s\right) E_{w}\left( s\right) ds \\ 
\leq C_{T,\delta }\left( \int_{t}^{t+T}\int_{\Omega }a\left( 1+s\right)
\left\vert \partial _{t}u\right\vert ^{2}dxds\right) +C_{T}\delta \left(
\int_{t}^{t+T}\int_{\Omega }\left\vert u\right\vert
^{2}dxds+\int_{t}^{t+T}\left( 1+s\right) E_{u}\left( s\right) ds\right) .%
\end{array}%
\end{equation*}%
On the other hand, using $\left( \ref{observability gradient 1+s 11}\right) $
we infer that%
\begin{eqnarray*}
&&\int_{t}^{t+T}\int_{\Omega }\left( 1+s\right) \left\vert \nabla \psi
\right\vert ^{2}\left\vert u\right\vert ^{2}dxds \\
&\leq &C_{T,\delta }\left( \int_{t}^{t+T}\int_{\Omega }a\left( 1+s\right)
\left\vert \partial _{t}u\right\vert ^{2}dxds\right) +C_{T}\delta \left(
\int_{t}^{t+T}\int_{\Omega }\left\vert u\right\vert
^{2}dxds+\int_{t}^{t+T}\left( 1+s\right) E_{u}\left( s\right) ds\right) .
\end{eqnarray*}%
Now $\left( \ref{1+t Xt estimate}\right) $ and the two estimates above gives%
\begin{equation*}
\begin{array}{l}
\left( 1+t+T\right) X\left( t+T\right) -\left( 1+t\right) X\left( t\right)
+\left( 1-3C_{T}\delta \right) \int_{t}^{t+T}\left( 1+s\right) E_{u}\left(
s\right) ds \\ 
+\left( k-\frac{2}{\epsilon _{0}}-3C_{T,\delta }\right)
\int_{t}^{t+T}\int_{\Omega }a\left( 1+s\right) \left\vert \partial
_{t}u\right\vert ^{2}dxds \\ 
\leq \int_{t}^{t+T}X\left( s\right) ds+C_{T}\delta
\int_{t}^{t+T}\int_{\Omega }\left\vert u\right\vert ^{2}dxds.%
\end{array}%
\end{equation*}%
We have%
\begin{eqnarray}
X\left( t\right) &\leq &\frac{3}{2}\left\Vert a\right\Vert _{\infty
}\int_{\Omega }\left\vert v\left( t\right) \right\vert ^{2}dx+\left( k+\frac{%
2}{\epsilon _{0}}\right) E_{u}\left( t\right) \text{ and}
\label{proof of theo 2 X t upper bound} \\
X\left( t\right) &\geq &\frac{\epsilon _{0}}{4}\int_{\Omega }\left\vert
v\left( t\right) \right\vert ^{2}dx+\left( k-\frac{8}{\epsilon _{0}}\right)
E_{u}\left( t\right) .  \label{proof of theore 2 Xt lower bound}
\end{eqnarray}%
We choose $\delta $ and $k$ such that 
\begin{align*}
1-3C_{T}\delta & =\frac{1}{2}, \\
k-\frac{2}{\epsilon _{0}}-3C_{T,\delta }& \geq \epsilon >0\text{ and }k-%
\frac{8}{\epsilon _{0}}\geq \epsilon .
\end{align*}%
Thus we get%
\begin{equation}
\begin{array}{l}
\left( 1+t+T\right) X\left( t+T\right) -\left( 1+t\right) X\left( t\right) +%
\frac{1}{2}\int_{t}^{t+T}\left( 1+s\right) E_{u}\left( s\right) ds \\ 
+\epsilon \int_{t}^{t+T}\int_{\Omega }a\left( 1+s\right) \left\vert \partial
_{t}u\right\vert ^{2}dxds \\ 
\leq \int_{t}^{t+T}X\left( s\right) ds+\frac{1}{6}\int_{t}^{t+T}\int_{\Omega
}\left\vert u\right\vert ^{2}dxds.%
\end{array}
\label{1+t Xt estimate final}
\end{equation}%
As in the proof of theorem 1, from the estimate above, we deduce that%
\begin{equation}
\begin{array}{l}
\underset{%
\mathbb{R}
_{+}}{\sup }\left( \left( 1+t\right) X\left( t\right) \right) +\frac{1}{2}%
\int_{0}^{+\infty }\left( 1+s\right) E_{u}\left( s\right) ds \\ 
\leq C\left( X\left( 0\right) +\int_{0}^{+\infty }X\left( s\right)
ds+\int_{0}^{+\infty }\int_{\Omega }\left\vert u\right\vert ^{2}dxds\right) .%
\end{array}
\label{proof of theo 2 Xt bound 1}
\end{equation}

Using the fact that $1-\psi \equiv 1$ for $\left\vert x\right\vert \geq 2L$,
we obtain%
\begin{equation*}
\underset{%
\mathbb{R}
_{+}}{\sup }\left( 1+t\right) \int_{\left\{ \left\vert x\right\vert \geq
2L\right\} }\left\vert u\left( t\right) \right\vert ^{2}dx\leq C\left(
X\left( 0\right) +\int_{0}^{+\infty }X\left( s\right) ds+\int_{0}^{+\infty
}\int_{\Omega }\left\vert u\right\vert ^{2}dxds\right) .
\end{equation*}%
On the other hand, using $\left( \ref{theorem 1 energy result}\right) ,$ we
obtain%
\begin{eqnarray*}
\left( 1+t\right) \int_{\Omega \cap B_{2L}}\left\vert u\left( t\right)
\right\vert ^{2}dx &\leq &C_{L}\left( 1+t\right) \int_{\Omega }\left\vert
\nabla u\left( t\right) \right\vert ^{2}dx \\
&\leq &C_{L}\left( 1+t\right) E_{u}\left( t\right) \\
&\leq &CI_{0}.
\end{eqnarray*}%
Combining the estimates above we deduce that%
\begin{equation}
\begin{array}{c}
\underset{%
\mathbb{R}
_{+}}{\sup }\left( 1+t\right) \int_{\Omega }\left\vert u\left( t\right)
\right\vert ^{2}dx\leq C\left( X\left( 0\right) +\int_{0}^{+\infty }X\left(
s\right) ds+\int_{0}^{+\infty }\int_{\Omega }\left\vert u\right\vert
^{2}dxds+I_{0}\right) .%
\end{array}
\label{1+t l2 u norm}
\end{equation}%
Now we need the following result due to Ikehata \cite[lemma 2.5]{ikehata}

\begin{lemma}
Let $u$ be a solution of $\left( \ref{system}\right) $ with initial data $%
\left( u_{0},u_{1}\right) $ in $H_{0}^{1}\left( \Omega \right) \times
L^{2}\left( \Omega \right) $ which satisfies%
\begin{equation*}
\left\Vert d\left( \cdot \right) \left( u_{1}+au_{0}\right) \right\Vert
_{L^{2}}<+\infty
\end{equation*}%
where 
\begin{equation*}
d\left( x\right) =\left\{ 
\begin{array}{ll}
\left\vert x\right\vert & d\geq 3 \\ 
\left\vert x\right\vert \ln \left( B\left\vert x\right\vert \right) & d=2%
\end{array}%
\right.
\end{equation*}%
with $B\underset{x\in \Omega }{\inf }\left\vert x\right\vert \geq 2$. Then
there exists $C>0$, such that%
\begin{equation}
\left\Vert u\left( t\right) \right\Vert
_{L^{2}}^{2}+\int_{0}^{t}\int_{\Omega }a\left\vert u\left( s,x\right)
\right\vert ^{2}dxds\leq C\left( \left\Vert u_{0}\right\Vert
_{L^{2}}^{2}+\left\Vert d\left( \cdot \right) \left( u_{1}+au_{0}\right)
\right\Vert _{L^{2}}^{2}\right) ,  \label{l2 bound ike}
\end{equation}%
for all $t\geq 0$.
\end{lemma}

We have 
\begin{eqnarray*}
\int_{0}^{t}\int_{\Omega }\left\vert u\right\vert ^{2}dxds &\leq
&\int_{0}^{t}\int_{\Omega \cap B_{L}}\left\vert u\right\vert
^{2}dxds+\int_{0}^{t}\int_{\left\{ \left\vert x\right\vert \geq L\right\}
}\left\vert u\right\vert ^{2}dxds \\
&\leq &C_{L}\int_{0}^{t}\int_{\Omega \cap B_{L}}\left\vert \nabla
u\right\vert ^{2}dxds+\frac{1}{\epsilon _{0}}\int_{0}^{t}\int_{\Omega
}a\left\vert u\left( s,x\right) \right\vert ^{2}dxds \\
&\leq &C_{L}\int_{0}^{t}E_{u}\left( s\right) ds+\frac{1}{\epsilon _{0}}%
\int_{0}^{t}\int_{\Omega }a\left\vert u\left( s,x\right) \right\vert
^{2}dxds.
\end{eqnarray*}%
Now using $\left( \ref{l2 bound ike}\right) $ and $\left( \ref{estimate
integral Eu t}\right) $, we get%
\begin{equation*}
\int_{0}^{+\infty }\int_{\Omega }\left\vert u\right\vert ^{2}dxds\leq CI_{1},
\end{equation*}%
with%
\begin{equation*}
I_{1}=\left\Vert u_{0}\right\Vert _{H^{1}}^{2}+\left\Vert u_{1}\right\Vert
_{L^{2}}^{2}+\left\Vert d\left( \cdot \right) \left( u_{1}+au_{0}\right)
\right\Vert _{L^{2}}^{2}.
\end{equation*}%
On the other hand, using $\left( \ref{proof of theo 2 X t upper bound}%
\right) $ and the fact that%
\begin{equation*}
\int_{0}^{+\infty }E_{u}\left( s\right) ds\leq CI_{0}
\end{equation*}%
\ we obtain%
\begin{eqnarray*}
\int_{0}^{+\infty }X\left( s\right) ds &\leq &C\left( \int_{0}^{+\infty
}\int_{\Omega }\left\vert u\right\vert ^{2}dxds+\int_{0}^{+\infty
}E_{u}\left( s\right) ds\right) \\
&\leq &CI_{1}.
\end{eqnarray*}%
Finally it is clear that, 
\begin{equation*}
X\left( 0\right) \leq CI_{1}.
\end{equation*}%
Therefore from $\left( \ref{1+t l2 u norm}\right) $%
\begin{eqnarray*}
\underset{%
\mathbb{R}
_{+}}{\sup }\left( 1+t\right) \int_{\Omega }\left\vert u\left( t\right)
\right\vert ^{2}dx &\leq &C\left( X\left( 0\right) +\int_{0}^{+\infty
}X\left( s\right) ds+\int_{0}^{+\infty }\int_{\Omega }\left\vert
u\right\vert ^{2}dxds+I_{0}\right) \\
&\leq &CI_{1},
\end{eqnarray*}%
and $\left( \ref{proof of theo 2 Xt bound 1}\right) ,$ gives%
\begin{equation*}
\int_{0}^{+\infty }\left( 1+s\right) E_{u}\left( s\right) ds\leq CI_{1}.
\end{equation*}%
\ The energy decay estimate follows from the fact that%
\begin{eqnarray*}
\left( 1+t\right) ^{2}E_{u}\left( t\right) &\leq &E_{u}\left( 0\right)
+2\int_{0}^{+\infty }\left( 1+s\right) E_{u}\left( s\right) ds \\
&\leq &CI_{1},\text{ for all }t\geq 0.
\end{eqnarray*}

This finishes the proof of theorem 2 and now we give the proof of
proposition 2.

\subsection{Proof of proposition 2}

Let $n\in 
\mathbb{N}
^{\ast }$ and $u$ the solution of $\left( \ref{system}\right) $ with initial
data in $D\left( A^{n}\right) $ such that $u_{0}\in L^{2}\left( \Omega
\right) $ and%
\begin{equation*}
\left\Vert d\left( \cdot \right) \left( u_{1}+au_{0}\right) \right\Vert
_{L^{2}}<+\infty .
\end{equation*}%
Let $v=\partial _{t}u.$ Using these estimates 
\begin{equation*}
\int_{0}^{+\infty }\left( 1+s\right) E_{u}\left( s\right) ds\leq CI_{1}\text{
and }\int_{0}^{+\infty }\left( 1+s\right) E_{v}\left( s\right) ds\leq
CI_{0,1},
\end{equation*}%
and proceeding as in the proof of proposition 1, we show that 
\begin{equation*}
\int_{0}^{+\infty }\left( 1+s\right) ^{2}E_{v}\left( s\right) ds\leq
C_{1}I_{1,1}.
\end{equation*}%
We set $u_{n}=\partial _{t}^{n}u,$ for $n\geq 1$. Using induction argument
and arguing as in the proof of proposition 1, we prove that 
\begin{equation*}
\int_{0}^{+\infty }\left( 1+s\right) ^{n+1}E_{u_{n}}\left( s\right) ds\leq
C_{n}I_{1,n},
\end{equation*}%
with%
\begin{equation*}
I_{1,n}=\sum_{i=0}^{n}\left\Vert A^{i}\left( u_{0},u_{1}\right) \right\Vert
_{H}^{2}+\left\Vert u_{0}\right\Vert _{L^{2}}^{2}+\left\Vert d\left( \cdot
\right) \left( u_{1}+au_{0}\right) \right\Vert _{L^{2}}^{2}.
\end{equation*}%
The energy decay estimate follows from the fact that%
\begin{eqnarray*}
\left( 1+t\right) ^{n+2}E_{u_{n}}\left( t\right)  &\leq &E_{u_{n}}\left(
0\right) +\left( n+2\right) \int_{0}^{t}\left( 1+s\right)
^{n+1}E_{u_{n}}\left( s\right) ds \\
&\leq &C_{n}I_{1,n},\text{ for all }t\geq 0.
\end{eqnarray*}%
Now using the estimate above, we infer that,%
\begin{eqnarray*}
\left( 1+t\right) ^{n+1}\left\Vert \partial _{t}^{n}u\left( t\right)
\right\Vert _{L^{2}}^{2} &\leq &2\left( 1+t\right) ^{n+1}E_{u_{n-1}}\left(
t\right)  \\
&\leq &C_{n-1}I_{1,n-1}\text{ for all }t\geq 0,\text{ }
\end{eqnarray*}

We have $\partial _{t}^{n-1}u$ is a solution of the following system 
\begin{equation}
\left\{ 
\begin{array}{lc}
\partial _{t}^{n+1}u-\Delta \partial _{t}^{n-1}u+a\left( x\right) \partial
_{t}^{n}u=0 & \text{in }\mathbb{R}_{+}\times \Omega , \\ 
\partial _{t}^{n-1}u=0 & \text{on }\mathbb{R}_{+}\times \Gamma , \\ 
\left( \partial _{t}^{n-1}u\left( 0,x\right) ,\partial _{t}^{n}u\left(
0,x\right) \right) \in D\left( A\right)  & 
\end{array}%
\right.   \label{sys dtnu 1}
\end{equation}%
therefore%
\begin{equation*}
\left( \partial _{t}^{n-1}u,\partial _{t}^{n}u\right) \in C^{1}\left( 
\mathbb{R}
_{+},H\right) .
\end{equation*}%
Using Eq $\left( \ref{sys dtnu 1}\right) ,$ we infer that%
\begin{eqnarray*}
\left( 1+t\right) ^{n+1}\left\Vert \Delta \partial _{t}^{n-1}u\left(
t\right) \right\Vert _{L^{2}}^{2} &\leq &C\left( 1+t\right) ^{n+1}\left(
E_{u_{n}}\left( t\right) +E_{u_{n-1}}\left( t\right) \right)  \\
&\leq &CI_{1,n}.
\end{eqnarray*}

\end{document}